\documentclass[11pt]{article}

\usepackage{amsmath}
\usepackage{amsfonts}
\usepackage{dsfont}
\usepackage{amssymb}
\usepackage{amsthm}
\usepackage{stmaryrd}
\usepackage{graphicx}
\usepackage[T1]{fontenc}
\usepackage[utf8]{inputenc}
\usepackage{babel}
\usepackage[left=2cm,right=2cm,top=2cm,bottom=3cm]{geometry}
\usepackage{dsfont}
\usepackage[svgnames]{xcolor}
\usepackage{enumitem}
\usepackage{hyperref}

\usepackage{caption}

\usepackage{tikz}
\usepackage[framemethod=TikZ]{mdframed}
\mdfsetup{skipabove=\topskip,skipbelow=\topskip}

\usetikzlibrary{patterns}
\usepackage{tkz-euclide}
\usepackage{xkeyval}
\usepackage{ifthen}

\usepackage[most]{tcolorbox}
\usepackage{blindtext}

\usepackage{ mathrsfs }

\newcommand{\elemcube}[4][white]{	
	\draw [fill=#1,very thin] (#2+1,#3,#4) -- ++(0,1,0) -- ++(0,0,-1) -- ++(0, -1, 0) -- cycle;
	\draw [fill=#1,very thin] (#2,#3+1,#4) -- ++(1,0,0) -- ++(0,0,-1) -- ++(-1, 0, 0) -- cycle; 
	\draw [fill=#1,very thin] (#2,#3,#4)   -- ++(1,0,0) -- ++(0,1,0)  -- ++(-1, 0, 0) -- cycle;                         
}

\newtheorem{theoreme}{Theorem}[section]
\newtheorem{lemme}[theoreme]{Lemma}
\theoremstyle{definition}
\newtheorem{corollaire}[theoreme]{Corollary}
\theoremstyle{definition}

\theoremstyle{definition}

\theoremstyle{definition}

\theoremstyle{definition}

\theoremstyle{definition}

\theoremstyle{definition}
\newtheorem{remarque}[theoreme]{Remark}
\theoremstyle{definition}
\newtheorem{proposition}[theoreme]{Proposition}
\theoremstyle{definition}

\theoremstyle{definition}
\newtheorem*{damier}{Construction of random checkerboards}
\theoremstyle{definition}
\newtheorem{claim}[theoreme]{Claim}

\addto\captionsfrench{}

\newcommand{\Z}{\mathbb{Z}}

\newcommand{\R}{\mathbb{R}}
\newcommand{\N}{\mathbb{N}}

\newcommand{\T}{\mathbb{T}}
\renewcommand{\P}{\mathbb{P}}

\newcommand{\cA}{\mathcal{A}}

\newcommand{\cD}{\mathcal{D}}

\newcommand{\cG}{\mathcal{G}}

\newcommand{\cP}{\mathscr{P}}

\newcommand{\tij}{t^i_j\left(\gamma\right)}
\newcommand{\tio}{t^i_0\left(\gamma\right)}
\newcommand{\tioji}{t^{i_0}_{j_i}\left(\gamma\right)}
\newcommand{\tiojii}{t^{i_0}_{j_i+1}\left(\gamma\right)}
\newcommand{\tioo}{t^{i_0}_{0}\left(\gamma\right)}
\newcommand{\tijp}{t^i_j\left(\gamma^\prime\right)}
\newcommand{\tiop}{t^i_0\left(\gamma^\prime\right)}
\newcommand{\tiojip}{t^{i_0}_{j_i}\left(\gamma^\prime\right)}
\newcommand{\tiojiip}{t^{i_0}_{j_i+1}\left(\gamma^\prime\right)}
\newcommand{\tioop}{t^{i_0}_{0}\left(\gamma^\prime\right)}
\newcommand{\ai}{\alpha^i\left(\gamma\right)}
\newcommand{\aio}{\alpha^{i_0}\left(\gamma\right)}
\newcommand{\aip}{\alpha^i\left(\gamma^\prime\right)}
\newcommand{\aiop}{\alpha^{i_0}\left(\gamma^\prime\right)}
\newcommand{\tik}{t^i_{k_i}\left(\gamma\right)}
\newcommand{\si}{s_i\left(\gamma\right)}
\newcommand{\sii}{s_{i+1}\left(\gamma\right)}

\newcommand{\tf}[2]{t^{#1}_{#2}\left(\gamma\right)}
\renewcommand{\sf}[1]{s_{#1}\left(\gamma\right)}

\newcommand{\m}{m^T_{\gamma}\left(\omega^n_{\varepsilon}\right)}
\newcommand{\mt}{m^T_{\gamma_t}\left(\omega^n_{\varepsilon}\right)}
\newcommand{\mpp}{m^T_{\gamma_p}\left(\omega^n_{\varepsilon}\right)}

\renewcommand{\mp}{m^T_{\gamma^\prime}\left(\omega^n_{\varepsilon}\right)}
\newcommand{\mk}{m^T_{\gamma_k}\left(\omega^n_{\varepsilon}\right)}

\newcommand{\mc}{m^T_{\gamma}\left(c_{i_1, \dots, i_d}\right)}
\newcommand{\mkc}{m^T_{\gamma_k}\left(c_{i_1, \dots, i_d}\right)}

\newcommand{\co}{\chi_{\omega_\varepsilon^n}}
\newcommand{\cc}{\chi_{c_{i_1, \dots, i_d}}}
\newcommand{\ccg}{\chi_{c_{i_1, \dots, i_d}}\left(\g\right)}
\newcommand{\ccgk}{\chi_{c_{i_1, \dots, i_d}}\left(\g_k\right)}
\newcommand{\ccgkt}{\chi_{c_{i_1, \dots, i_d}}\left(\g_k\left(t\right)\right)}
\newcommand{\ccgt}{\chi_{c_{i_1, \dots, i_d}}\left(\g\left(t\right)\right)}
\newcommand{\cij}{c_{i_1, \dots, i_d}}
\newcommand{\icij}{\mathring{c}_{i_1, \dots, i_d}}

\newcommand{\g}{\gamma}
\newcommand{\G}{\hat{\Gamma}}

\title{Geometrical quantity on random checkerboards on the regular torus}
\author{Léa Gohier\footnote{Institut Denis Poisson, UMR-CNRS 7013, Université de Tours, lea.gohier@univ-tours.fr}}
\date{}

\allowdisplaybreaks[1]

\begin{document}
	
	\maketitle
	
	\pagenumbering{arabic}
	


\begin{abstract}
	In the study of the observability of the wave equation (here on $(0,T)\times \T^d$, where $\T^d$ is the d-dimensional torus), a condition naturally emerges as a sufficient observability condition. This condition, which writes $\ell^T\left(\omega\right)>0$, signifies that the smallest time spent by a geodesic in the subset $\omega\subset \T^d$ during time $T$ is non-zero. In other words, the subset $\omega$ detects any geodesic propagating on the d-dimensional torus during time $T$. Here, the subset $\omega$ is randomly defined by drawing a grid of $n^d$, $n\in\N$, small cubes of equal size and by adding them to $\omega$ with probability $\varepsilon>0$.
	In this article, we establish a probabilistic property of the functional $\ell^T$: the random law $\ell^T\left(\omega_\varepsilon^n\right)$ converges in probability to $\varepsilon$ as $n \to + \infty$.
	Considering random subsets $\omega_\varepsilon^n$ allows us to construct subsets $\omega$ such that $\ell^T\left(\omega\right)=|\omega|$.
\end{abstract}
	
	\section{Introduction}

		\subsection*{Motivation}

	Let $(M,g)$ a compact Riemannian manifold. We denote by $\Gamma$ the set of geodesics propagating on $M$, that is,\label{geo} the set of projections onto $M$ of Riemannian geodesic curves in the co-sphere bundle $S^{*}M$.
	
	We consider the wave equation
	
	\begin{equation}\label{wave}
		\partial_{tt} y - \Delta_g y = 0 ;
	\end{equation}
	on $\left(0,T\right) \times M$, where $\Delta_g$ is the Laplace-Beltrami operator on $M$ with respect of the metric $g$.
	
	Denoting by $\text{d}x_g$ the canonical Riemannian volume, we define the observability constant $C_T\left(\omega\right) \geqslant 0$, where $\omega$ is a measurable subset of $M$, as the largest nonnegative constant $C$ such that the inequality 
	\begin{equation}
		\int_0^T \int_{\omega} |y\left(t,x\right)|^2 \text{d}x_g \text{d}t \geqslant C \left( \|y\left(0,\cdot\right)\|_{L^2\left(M\right)}^2 + \|\partial_t y\left(0,\cdot\right)\|_{H^{-1}\left(M\right)}^2 \right),
	\end{equation}
	where $H^{-1}\left(M\right)$ is the dual space of $H^{1}\left(M\right)$ with respect to the pivot space $L^2\left(M\right)$, is satisfied for any solution $y$ of the wave equation \eqref{wave}. Thus, 
	\begin{equation*}
		C_T\left(\omega\right)=\inf \left\{\int_0^T \int_{\omega} |y\left(t,x\right)|^2 \text{d}x_g \text{d}t \; \middle| \; \|\left(y\left(0,\cdot\right), \partial_t y\left(0,\cdot\right)\right)\|_{L^2\left(M\right)\times H^{-1}\left(M\right)}=1\right\}.
	\end{equation*}
	
	The wave equation \eqref{wave} is said to be observable on $\omega$ in time $T$ when $C_T\left(\omega\right)>0$. Considering the definition of $C_T\left(\omega\right)$, studying the observability of the wave equation appears to be difficult. This raises the question of the existence of sufficient observability conditions. The object $\ell^T\left(\omega\right)$ defined below precisely provides such a condition, through Theorem \ref{thobs}.
	
	Let $T>0$, let $\omega$ be a measurable subset of $M$, we define 
	\begin{equation}
		\ell^T\left(\omega\right) = \inf\limits_{\g\in\Gamma} \frac{1}{T} \int_0^T\chi_{\omega}\left(\g\left(t\right)\right) \text{d}t \text{ ; where } \chi_\omega : \begin{cases}
			M \longrightarrow \left\{0,1\right\}\\
			x \longmapsto \mathds{1}_{\omega}\left(x\right).
		\end{cases} 
	\end{equation} 
	Then, for any $\g\in\Gamma$, we define $m_\g^T\left(\omega\right) = \frac{1}{T} \int_0^T\chi_{\omega}\left(\g\left(t\right)\right) \text{d}t$. Thus, $\ell^T\left(\omega\right) = \inf\limits_{\g\in\Gamma} m_\g^T\left(\omega\right)$.
	The real number $\ell^T\left(\omega\right)$ represents the smallest amount of time a geodesic from $\Gamma$ spends inside $\omega$ between $0$ and $T$. 
	At least when $\omega$ is open, the condition $\ell^T\left(\omega\right)>0$ means that every geodesic from $\Gamma$ intersects $\omega$ during the time interval $T$. This condition is called the Geometric Control Condition. 
	The following theorem, proved in \cite{4}, provides a sufficient condition for the observability of the wave equation.
	
	\begin{theoreme}\label{thobs}
		If $\omega$ is open and $\ell^T\left(\omega\right)>0$, then the wave equation is observable on $\omega$ in time $T$.
	\end{theoreme}
	
	\begin{remarque}
		A similar result on manifolds with boundary is proved in \cite{3}. In this case, the rays are reflected at the boundary according to the laws of geometric optics.
	\end{remarque}
	
	Thus, $\ell^T\left(\omega\right)$ plays a crucial role in problems related to the observability of the wave equation, through the study of its positivity. The following theorem, proved in \cite{5}, also justifies the study of its exact value:
	\begin{theoreme}[Large-time observability]
		Given any $T>0$ and any measurable subset $\omega \subset M$, there exists $\alpha^T\in \left[\frac{1}{2} \ell_T\left(\mathring{\omega}\right), \frac{1}{2} \ell_T\left(\overline{\omega}\right)\right]$ such that the limit $\alpha=\lim\limits_{T\to +\infty} \alpha^T$ exists and we have
		\begin{equation}
			\underset{T\to + \infty}{\lim} \frac{C_T\left(\omega\right)}{T}=\min \left(\frac{1}{2} g_1\left(\omega\right), \alpha \right),
		\end{equation}
		where $g_1\left(\omega\right) = \inf\limits_{\phi} \frac{\int_\omega \left|\phi\left(x\right)\right|^2 \text{d}v_g}{\int_M \left|\phi\left(x\right)\right|^2 \text{d}v_g}$ (where the infimum runs over the set of all nonconstant eigenfunctions $\phi$ of $-\Delta_g$, and where $v_g$ is the canonical Riemannian volume on $M$).
	\end{theoreme}
	
	\begin{remarque}
		Except for pathological examples, the values of $\ell_T\left(\mathring{\omega}\right)$ and $\ell_T\left(\overline{\omega}\right)$ are equal.
	\end{remarque}
	
	Let $d\in\N^*$. In the entire remainder of the article, $M$ will refer to the torus in dimension $d$, $\T^d=\R^d/\Z^d$.

	\subsection*{Origin of the model}
	In this subsection, $d=2$.	
	The article of Hébrard and Humbert (\cite{6}) gives an algorithm to compute explicitly $\ell^T\left(\omega\right)$ when $\omega$ is a finite union of squares (which will be random later), and when $T\to +\infty$.
	
	\begin{center}
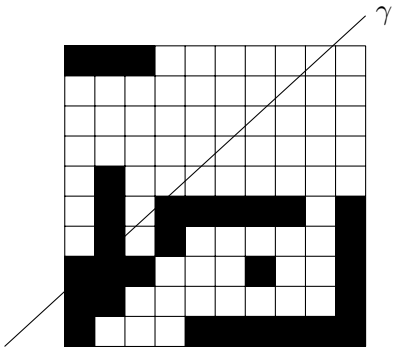
		
		\begin{tikzpicture}[scale=0.4]
			\fill (0,0) rectangle (1,3);
			\fill (1,1) rectangle (2,6);
			\fill (2,2) rectangle (3,3);
			\fill (3,3) rectangle (4,5);
			\fill (4,4) rectangle (8,5);
			\fill (10,5) rectangle (9,1);
			\fill (0,10) rectangle (3,9);
			\fill (4,0) rectangle (10,1);
			\fill (6,2) rectangle (7,3);
			\draw (0,0) grid (10,10);
			\draw (-2,0)--(10,11) node[right]{$\g$};
		\end{tikzpicture}
		\captionof{figure}{Example of $\omega$ (union of black squares).} 
	\end{center}

	The following propositions provide an upper bound for $\ell^T\left(\omega\right)$ for any $T>0$.
	
	\begin{proposition}
		Let $\g\in\Gamma$, dense in the torus $M$, i.e. $\left\{\g\left(t\right) \, \middle| \, t\in [0,+\infty[ \right\}$ is dense in the torus, then, $\m \underset{T \to +\infty}{\longrightarrow} \left|\omega\right|$.
	\end{proposition}
	
	\begin{proposition}
		If $\omega\subset M$ is Riemann-integrable, then $\underset{T\to +\infty}{\limsup} \; \ell^T\left(\omega\right) \leqslant \left|\omega\right|$.
	\end{proposition}
	
	And then, for all $ T>0$, $\ell^T\left(\omega\right) \leqslant \left|\omega\right|$.
	This leads to the question that the article \cite{1} answers using a probabilistic method: Have we $\sup\limits_{\left|\omega\right|=\alpha} \ell^T\left(\omega\right) = \alpha$ ?
	
	
	In this article, we are more interested in studying the behaviour of $\ell^T\left(\omega\right)$ when $\omega$ is a random domain than just answering a similar question in the case of a $d$-dimensional torus. Indeed, we are currently studying a method that seems more effective for providing a general answer to this type of question.

	\subsection*{Presentation of the model}
	
	We want to compute the probability that $\ell^T\left(\omega\right)>0$ on the $d$-dimensional torus, where $\omega$ is a randomly constructed subset of $M$. The subset $\omega$ can be seen as a collection of sensors placed randomly on $M$, and 
	$\ell^T\left(\omega\right)>0$ represents the ability of $\omega$ to detect all the geodesics living on $M$.
	To do this, we define $\omega$ according to the following model:
	
	Let $n\in\N^*$, we consider a regular grid $\cG^n=\left(c_{i_1, \dots, i_d}\right)_{1\leqslant i_1, \dots, i_d\leqslant n}$ in the cube $\left[0, 1\right]^d$, consisting of $n^d$ cells. We have $\left[0, 1\right]^d=\bigcup\limits_{i_1, \dots, i_d=1}^n c_{i_1, \dots, i_d}$ where, for every $\left(i_1, \dots, i_d\right)\in \llbracket 1, n\rrbracket^d$, $c_{i_1, \dots, i_d}= \left[\left(i_1-1\right)/n, i_1/n\right]\times \dots \times \left[\left(i_d-1\right)/n, i_d/n\right]$. For every $\left(i_1^\prime, \dots, i_d^\prime\right)\in \Z^d$, we identify the cube $c_{i_1^\prime, \dots, i_d^\prime}$ with the cube $c_{i_1, \dots, i_d}$ of the grid if: for all $j\in \llbracket 1, d \rrbracket, \; i_j\equiv i_j^\prime \left[n\right] $.
	
	\begin{damier}\label{dam}
		Let $\varepsilon\in\left[0,1\right]$. We consider the grid $\cG^n$ initially white. We randomly blacken certain cells of the grid $\cG^n$ according to the following procedure: for every $\left(i_1, \dots, i_d\right)\in \llbracket 0, 1 \rrbracket^d$, we blacken the cell $c_{i_1, \dots, i_d}$ of the grid with probability $\varepsilon$. All choices are assumed to be mutually independent. In other words, we select the cells to blacken in the grid $\cG^n$ by considering $n^d$ independent Bernoulli random variables denoted $\left(X_{i_1, \dots, i_d}\right)_{1\leqslant i_1, \dots, i_d\leqslant n}$, with parameter $\varepsilon$. We denote by $\omega_\varepsilon^n$ the subset of $\left[0,1\right]^d$, defined as the union of all randomly blackened cells.
	\end{damier}

	\begin{center}
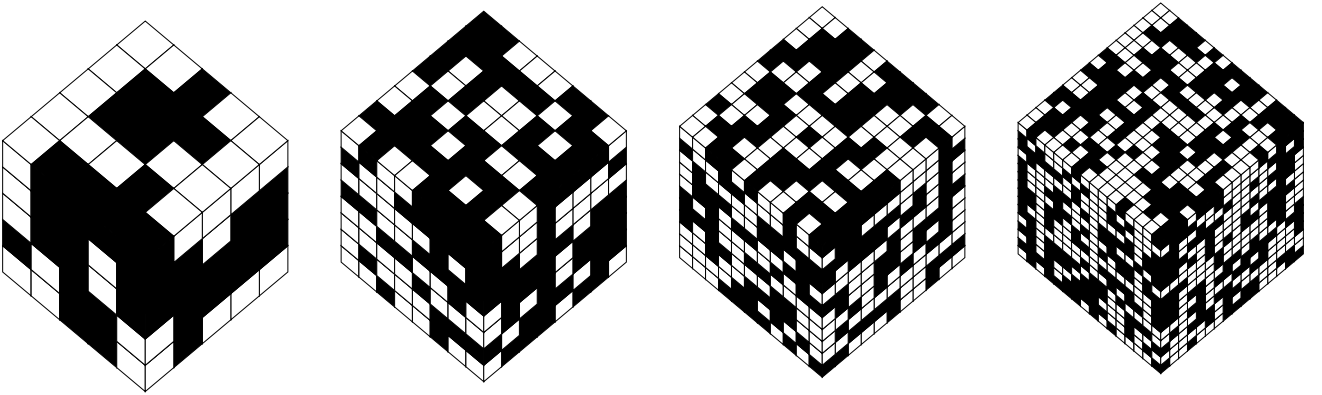

		\begin{tikzpicture}[x=(220:0.5cm), y=(-40:0.5cm), z=(90:0.353cm), scale=0.45]
			\foreach \k in {0,1,...,10} {\foreach \l in {0,1,...,10} {\foreach \n in {0,1,...,10} {\pgfmathsetmacro{\x}{random(0,1)} 
						\ifthenelse{\x<1}{\elemcube[black]{\n}{\l}{\k}}{\elemcube{\n}{\l}{\k}}}}}
			\begin{scope}[xshift=-10cm, scale=1.375]
				\foreach \k in {0,1,...,7} {\foreach \l in {0,1,...,7} {\foreach \n in {0,1,...,7} {\pgfmathsetmacro{\x}{random(0,1)} 
							\ifthenelse{\x<1}{\elemcube[black]{\n}{\l}{\k}}{\elemcube{\n}{\l}{\k}}}}}
			\end{scope}
			\begin{scope}[xshift=10cm, scale=0.6875]
				\foreach \k in {0,1,...,15} {\foreach \l in {0,1,...,15} {\foreach \n in {0,1,...,15} {\pgfmathsetmacro{\x}{random(0,1)} 
							\ifthenelse{\x<1}{\elemcube[black]{\n}{\l}{\k}}{\elemcube{\n}{\l}{\k}}}}}
			\end{scope}
			\begin{scope}[xshift=-20cm, scale=2.2]
				\foreach \k in {0,1,...,4} {\foreach \l in {0,1,...,4} {\foreach \n in {0,1,...,4} {\pgfmathsetmacro{\x}{random(0,1)} 
							\ifthenelse{\x<1}{\elemcube[black]{\n}{\l}{\k}}{\elemcube{\n}{\l}{\k}}}}}
			\end{scope}
		\end{tikzpicture}
		\captionof{figure}{Examples of random checkerboards in dimension $d=3$. $\omega_\varepsilon^n$ is the union of the black cells. From left to right : $n=5$, $n=8$, $n=11$, $n=16$, and $\varepsilon=\frac{1}{2}$.}
	\end{center}
	
	Let $\g\in \Gamma$, for every $\left(i_1, \dots, i_d\right)\in\left\{1, \dots, n\right\}^d$, we denote by $t_{i_1, \dots, i_d}\left(\g\right)$ the time spent by $\g$ in the cell $c_{i_1, \dots, i_d}$ of the grid $\cG^n$. We have $\sum\limits_{i_1, \dots, i_d=1}^n t_{i_1, \dots, i_d}\left(\g\right)=T$. Moreover, $\g$ can traverse the cell $c_{i_1, \dots, i_d}$ at most $T+1$ times, and within this cell, the maximum time spent by a geodesic during one passage is $\frac{\sqrt{d}}{n}$. Thus, we have:
	
	\begin{equation}
		\forall \left(i_1, \dots, i_d\right)\in\left\{1, \dots, n\right\}^d, \quad \frac{t_{i_1, \dots, i_d}\left(\g\right)}{T}\leqslant \frac{T+1}{T}\frac{\sqrt{d}}{n} \; ; \quad \text{and} \quad \sum\limits_{i_1, \dots, i_d=1}^n \frac{t_{i_1, \dots, i_d}\left(\g\right)}{T}=1 .
	\end{equation}

	\subsection*{Main result}

	\begin{theoreme}\label{th}
		Let $T>0$, let $\varepsilon \in [0,1]$. The random variable $\ell^T(\omega_\varepsilon^n)$ converges in probability to $\varepsilon$ as $n\to +\infty$. In other words, for any $\delta>0$, $\underset{n \to +\infty}{\lim}\mathbb{P}\left(\left|\ell^T\left(\omega_\varepsilon^n\right)-\varepsilon\right|\geqslant\delta\right)=0$.
	\end{theoreme}

\begin{remarque}
	The chosen model is particularly suitable for the torus, and the mathematical problem it poses is interesting. However, if one were to attempt to extend the result to other domains, potentially less regular, such a model may not generalize well. Another model using a random domain constructed via a Poisson point process could provide a more generalizable approach. This will be the subject of future research.
\end{remarque}
	
	To prove this theorem, we need to show that for any $T>0$ and any $\varepsilon\in[0,1]$, we have, for any $\delta>0$:
	\begin{enumerate}[label=(\roman*)]
		\item $\underset{n \to +\infty}{\lim}\mathbb{P}\left(\ell^T\left(\omega_\varepsilon^n\right)\geqslant\varepsilon+\delta\right)=0$;\label{thi}
		\item $\underset{n \to +\infty}{\lim}\mathbb{P}\left(\ell^T\left(\omega_\varepsilon^n\right)\leqslant\varepsilon-\delta\right)=0$.\label{thii}
	\end{enumerate}
	
	Furthermore, for each $T^\prime>0$, $\ell^{T^\prime}\left(\omega_\varepsilon^n\right) \leqslant \underset{T \to +\infty}{\lim} \ell^T\left(\omega_\varepsilon^n\right)$ (\cite{5}: the mapping $T\mapsto\ell^T\left(\omega_\varepsilon^n\right)$ is nonnegative, is bounded above by $1$ and is subadditive). Thus, the result remains true as $T \to + \infty$. From this, we deduce the following corollary:
	
	\begin{corollaire}
		Let $\varepsilon \in [0,1]$. For any $\delta>0$, $\underset{n \to +\infty}{\lim}\mathbb{P}\left(\left|\underset{T \to +\infty}{\lim} \ell^T(\omega_\varepsilon^n)-\varepsilon\right|\geqslant\delta\right)=0$.
	\end{corollaire}

	\section{Proof of \ref{thi}}

	The difficulty of proving Theorem \ref{th} lies in proving \ref{thii}.
	The proof of \ref{thi} is straightforward, and here are the details:  
	
	\begin{proof}[Proof of \ref{thi}]
		Let $\g\in \Gamma$. We have
		\begin{equation}\label{summ}
			\ell^T\left(\omega_\varepsilon^n\right)\leqslant \m= \sum\limits_{i_1, \dots, i_d=1}^n \frac{t_{i_1, \dots, i_d}\left(\g\right)}{T} X_{i_1, \dots, i_d}.
		\end{equation}
		Thus, $\P\left(\ell^T\left(\omega_\varepsilon^n\right)\geqslant\varepsilon+\delta\right)$ $\leqslant \P\left(\m \geqslant\varepsilon+\delta\right) \leqslant \P\left(\left|\m-\varepsilon\right|\geqslant\delta\right)$. Then, by applying Tchebychev's inequality to the random variable $\m$, we obtain:
		\begin{equation*}
			\P\left(\left|\m-\varepsilon\right|\geqslant\delta\right)\leqslant\frac{\text{Var}\left(\m\right)}{\delta^2}\leqslant \frac{1}{\delta^2}\frac{T+1}{T}\frac{\varepsilon\left(1-\varepsilon\right)\sqrt{d}}{n}.
		\end{equation*}
		Indeed, $\text{Var}\left(\m\right)= \sum\limits_{i_1, \dots, i_d=1}^n \left(\frac{t_{i_1, \dots, i_d}\left(\g\right)}{T}\right)^2 \text{Var}\left( X_{i_1, \dots, i_d}\right)$, due to the independence of $\left(X_{i_1, \dots, i_d}\right)_{1\leqslant i_1, \dots, i_d\leqslant n}$. Moreover, for any  $\left(i_1, \dots, i_d\right)\in\left\{1, \dots, n\right\}^d$, $\text{Var}\left(X_{i_1, \dots, i_d}\right)=\varepsilon\left(1-\varepsilon\right)$. Hence,
		\begin{equation*}
			\text{Var}\left(\m\right)=\varepsilon\left(1-\varepsilon\right) \sum\limits_{i_1, \dots, i_d=1}^n \left(\frac{t_{i_1, \dots, i_d}\left(\g\right)}{T}\right)^2.
		\end{equation*}
		Now, for any $\left(i_1, \dots, i_d\right)\in\left\{1, \dots, n\right\}^d$, $\frac{t_{i_1, \dots, i_d}\left(\g\right)}{T}\leqslant \frac{T+1}{T}\frac{\sqrt{d}}{n}$, and $\sum\limits_{i_1, \dots, i_d=1}^n \frac{t_{i_1, \dots, i_d}\left(\g\right)}{T}=1$. Thus,
		\begin{equation*}
			\text{Var}\left(\m\right)\leqslant\varepsilon\left(1-\varepsilon\right)\frac{T+1}{T}\frac{\sqrt{d}}{n}.
		\end{equation*}
		Hence the result.
	\end{proof}
	
	\section{Outline of the proof of \ref{thii}}
	
	
	
	In the proof of Theorem \ref{th}, we use a large deviation result (Proposition \ref{gd} in the following) proven in \cite{1} (Section 3.3, Proposition 1), where the difference with the classical large deviation result is that the constants $\lambda_i$ are not necessary equal to $\frac{1}{m}$.
	
	\begin{proposition}\label{gd}
		Let $m\geqslant 3$ be an integer, and let $\left(\lambda_1,\dots,\lambda_m\right)$ be an $m$-tuple of positive real numbers satisfying $\sum\limits_{i=1}^{m}\lambda_i=1$. Assume there exists $c>1$ such that for all $i$, $\lambda_i\leqslant \frac{c}{m}$. We define $Y_m=\sum\limits_{i=1}^{m}\lambda_i X_i$, where $\left(X_1,\dots,X_m\right)$ is an $m$-tuple of independent and identically distributed random variables with a common expectation $\varepsilon\in\left[0,1\right]$. Then, for any $\delta>0$, there exists $C_{\varepsilon,\delta}>0$ such that 
		\begin{equation*}
			\P\left(|Y_m-\varepsilon|\geqslant \delta\right)\leqslant C_{\varepsilon,\delta}\exp\left(-\delta^2\frac{m}{c}\right).
		\end{equation*} 
	\end{proposition}
	
	The structure of the proof of Theorem \ref{th} is inspired by \cite{1}: we restrict the study of $\m$ to a polynomial subfamily of geodesics $\g$. 
	However, both the subfamily of geodesics and the arguments allowing the study of $\ell^T\left(\omega_{\varepsilon}^n\right)$ are very different from what is presented in \cite{1}.
	If we were to reuse the same arguments as in the proof of \cite{1}, we would quickly encounter difficulties in adapting them to higher dimensions. The following paragraph gives the main steps of the proof of Theorem \ref{th}, and the next one compares it to the proof in dimension $2$ (\cite{1}).
	
		\subsection{Sketch of the proof of \ref{thii}}
	
	The main idea of the proof of \ref{thii} is to decompose the sum over all cells in the definition of $\m$ (see \eqref{summ}) into $d$ sums according to the types of hyperplanes (we will say that a hyperplane is of type $i \in \N$ if it is orthogonal to the $i$-th vector of the canonical basis of $\R^d$, see Figure \ref{hyp}). 
	\begin{center}
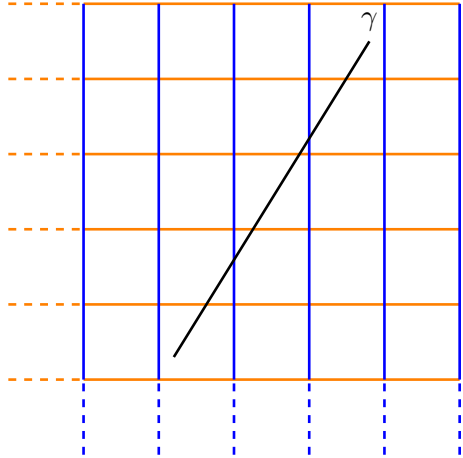

		\begin{tikzpicture}
			\draw[orange][line width=1pt] (0,1) -- (5,1);
			\draw[orange, dashed][line width=1pt] (-1,1) -- (0,1);
			\draw[orange][line width=1pt] (0,2) -- (5,2);
			\draw[orange, dashed][line width=1pt] (-1,2) -- (0,2);
			\draw[orange][line width=1pt] (0,3) -- (5,3);
			\draw[orange, dashed][line width=1pt] (-1,3) -- (0,3);
			\draw[orange][line width=1pt] (0,4) -- (5,4);
			\draw[orange, dashed][line width=1pt] (-1,4) -- (0,4);
			\draw[orange][line width=1pt] (0,0) -- (5,0);
			\draw[orange, dashed][line width=1pt] (-1,0) -- (0,0);
			\draw[orange][line width=1pt] (0,5) -- (5,5);
			\draw[orange, dashed][line width=1pt] (-1,5) -- (0,5);
			\draw[blue][line width=1pt] (2,0) -- (2,5);
			\draw[blue, dashed][line width=1pt] (2,-1) -- (2,0);
			\draw[blue][line width=1pt] (3,0) -- (3,5);
			\draw[blue, dashed][line width=1pt] (3,-1) -- (3,0);
			\draw[blue][line width=1pt] (1,0) -- (1,5);
			\draw[blue, dashed][line width=1pt] (1,-1) -- (1,0);
			\draw[blue][line width=1pt] (4,0) -- (4,5);
			\draw[blue, dashed][line width=1pt] (4,-1) -- (4,0);
			\draw[blue][line width=1pt] (0,0) -- (0,5);
			\draw[blue, dashed][line width=1pt] (0,-1) -- (0,0);
			\draw[blue][line width=1pt] (5,0) -- (5,5);
			\draw[blue, dashed][line width=1pt] (5,-1) -- (5,0);
			\draw[line width=1pt] (1.2,0.3) -- (3.8,4.5) node[above]{$\g$};
		\end{tikzpicture}
		\captionof{figure}{In dimension $d=2$: the hyperplanes of type $1$ are represented in blue (vertical), those of type $2$ in orange (horizontal).}\label{hyp}
	\end{center}
	Indeed, the key property, based on Thales' theorem, is that between two consecutive hyperplanes of the same type, a geodesic spends a constant amount of time.
	We want to evaluate the infimum of $\m$ over the set of geodesic rays: 
	
	\begin{enumerate}
		\item We start by excluding the geodesics in $\Gamma$ that are completely contained within a hyperplane $\Gamma_H$ in order to work on a subset of $\Gamma$ where $\g\mapsto\m$ is continuous.
		\item We partition the set of geodesic rays $\Gamma\setminus\Gamma_H$ living on $M$ into $\#\cP_d$ classes of geodesics $\Gamma_p$, such that geodesics in the same class encounter the same cells of the grid $\cG^n$ in the same order (see the precise definition of the set $\cP_d$ in the beginning of Section \ref{secpr} and Figure \ref{defclas}). We are then left with the task of evaluating the infimum of $\m$ over the geodesics in the same class.
		\begin{center}
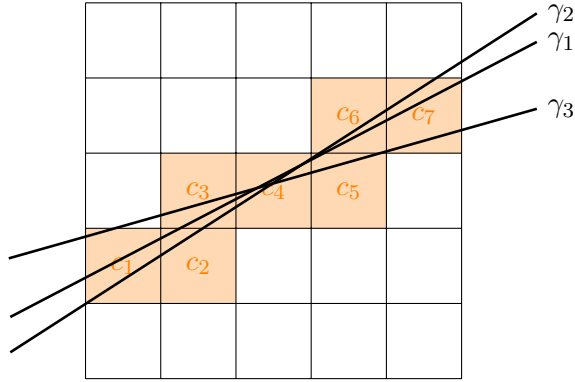

			\begin{tikzpicture}
				\fill[color=orange!30] (0,1) rectangle (1,2);
				\fill[color=orange!30] (1,1) rectangle (2,2);
				\fill[color=orange!30] (1,2) rectangle (2,3);
				\fill[color=orange!30] (2,2) rectangle (3,3);
				\fill[color=orange!30] (3,2) rectangle (4,3);
				\fill[color=orange!30] (3,3) rectangle (4,4);
				\fill[color=orange!30] (4,3) rectangle (5,4);
				\draw[orange] (0.5,1.5) node{$c_1$};
				\draw[orange] (1.5,1.5) node{$c_2$};
				\draw[orange] (1.5,2.5) node{$c_3$};
				\draw[orange] (2.5,2.5) node{$c_4$};
				\draw[orange] (3.5,2.5) node{$c_5$};
				\draw[orange] (3.5,3.5) node{$c_6$};
				\draw[orange] (4.5,3.5) node{$c_7$};
				\draw (0,0) grid (5,5);
				\draw[line width=1pt] (-1,0.82)--(6,4.48) node[right]{$\g_1$};
				\draw[line width=1pt] (-1,0.35)--(6,4.86) node[right]{$\g_2$};
				\draw[line width=1pt] (-1.02,1.6)--(6,3.59) node[right]{$\g_3$};
			\end{tikzpicture}
			\captionof{figure}{Examples of geodesics: $\g_1$ and $\g_2$ are in the same class, not $\g_3$.}\label{defclas}
		\end{center}
		\item We evaluate the difference in time spent in $\omega$ by two geodesics from the same class: we rearrange the sum in the definition of $\m$ to handle the different types of encountered hyperplanes separately. In fact, we treat them differently depending on the number of occurrences of each type of hyperplane: if a type $i$ hyperplane is encountered only a few times by the geodesics in the considered class, then, in probability, the impact of this hyperplane type in evaluating $|\m-\mp|$ will be negligible. On the contrary, if a type $i$ hyperplane is encountered a larger number of times, we can show, by applying Proposition \ref{gd} (large deviation result) to the appropriate random variables, that in probability, the impact of this hyperplane type in evaluating $|\m-\mp|$ will be exponentially decaying in $-\sqrt{n}$.
		
		By treating things in this way, we avoid confronting pathological cases like the one below, whose probability of occurrence is extremely low but which maximizes the difference $|\m-\mp|$.

		\begin{center}
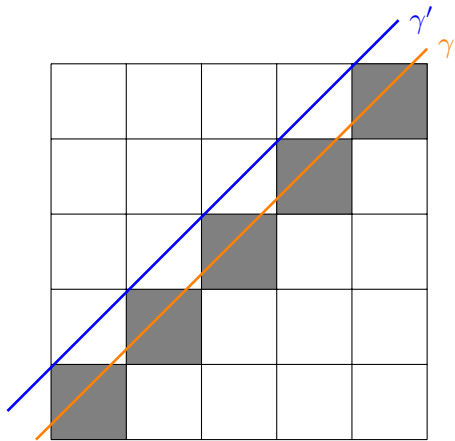
		
			\begin{tikzpicture}
				\fill[color=black!50] (0,0) rectangle (1,1);
				\fill[color=black!50] (1,1) rectangle (2,2);
				\fill[color=black!50] (2,2) rectangle (3,3);
				\fill[color=black!50] (3,3) rectangle (4,4);
				\fill[color=black!50] (4,4) rectangle (5,5);
				\draw (0,0) grid (5,5);
				\draw[color=orange][line width=1 pt] (-0.2,0)--(5,5.2) node[right]{$\g$};
				\draw[color=blue][line width=1 pt] (-0.58,0.38)--(4.62,5.58) node[right]{$\g'$};
			\end{tikzpicture}
			\captionof{figure}{Example of $\g$ and $\g'$ in the same class and where $\m \gg m_{\g^\prime}^T\left(\omega_{\varepsilon}^n\right)$.} 
		\end{center}
		
		\item Finally, we apply Proposition \ref{gd} to a representative of each class. The previous steps then allow us to reduce the problem to bounding $\P\left(\inf\limits_{\g \in \Gamma\backslash\Gamma_H} \m \leqslant\varepsilon-\delta\right)$ by
		\begin{equation*}
			\#\cP_d \left[C_2\exp\left(\frac{-\delta^2nT^2}{4d\left(T+1\right)}\right)+ C_1 \exp\left(\frac{-\delta^2\sqrt{nT}}{2^{14}d^{\frac{17}{4}}}\right)\right] ,
		\end{equation*}
		where $C_1$ and $C_2$ are strictly positive constants.
		To conclude, it remains to show that the number of classes $\#\cP_d$ is polynomial in $n$. This point of the proof boils down to solving an interesting algebraic problem: bounding the number of ways to separate $n$ points in the plane with a line.
	\end{enumerate}

	\subsection{Comparison with the proof in dimension $d=2$}
	
	As done in dimension $d=2$, our goal is to reduce the study of $\m$ (time spent by a geodesic $\g$ in $\omega_{\varepsilon}^n$, see paragraph "Construction of random checkerboards" page \pageref{dam}) to a subfamily of geodesics (definition page \pageref{geo}) with a polynomial cardinality in $n$, and then apply the large deviation result to the geodesics in this subfamily.
	In dimension $d=2$ (see \cite{1}), the relevant subfamily consists of geodesics passing through two corners, and it is possible to control the difference between the infimum over the set of all geodesics and the infimum over the geodesics passing through two corners. In higher dimensions, we encounter difficulties in generalizing the proof due to the following reason: we start with an arbitrary geodesic, and using only translations, we aim to bring it to a geodesic passing through a corner;
	
	\begin{itemize}
		\item In dimension $d=2$, the vertices of the grid $\cG^n$ are projected orthogonally to $\g$, which ensures that when a corner is encountered in the projection, a corner is indeed encountered in the original sense given in the definition, as seen in the Figure \ref{dim2}.
		Thus, both $\g$ and its translation $\g_s$ pass through the same cells until $\g_s$ intersects a corner.
		Therefore, the function $s\longmapsto m^T_{\g_s}\left(\omega\right)=\sum\limits_{\left(i, j\right)\in I\left(\g\right)} \frac{t^n_{ij}\left(\g_s\right)}{T}X^n_{ij}$ is affine, which is a key point in the proof.
		
		\begin{center}
			\begin{tikzpicture}[scale=2]\label{dim2}

				\draw (0,0)--(2,0);
				\draw (0,1)--(2,1);
				\draw (0,2)--(2,2);
				\draw (0,0)--(0,2);
				\draw (1,0)--(1,2);
				\draw (2,0)--(2,2);
				
				\draw[blue][opacity=0.2] (-2.82,0.32)--(0.33,-2.43);

				\draw[orange] (0,0)--(-1.06,-1.22);
				\draw[orange] (1,0)--(-0.49,-1.71);
				\draw[orange] (2,0)--(0.08,-2.21);
				\draw[orange] (0,1)--(-1.56,-0.79);
				\draw[orange] (1,1)--(-0.99,-1.28);
				\draw[orange] (2,1)--(-0.42,-1.78);
				\draw[orange] (0,2)--(-2.05,-0.36);
				\draw[orange] (1,2)--(-1.48,-0.85);
				\draw[orange] (2,2)--(-0.91,-1.35);
				
				\draw[line width=1.5pt] (0.63,0)--(2,1.58);
				\draw[dashed, line width=1.5pt] (-0.7,-1.53)--(0.63,0);
				\draw[dashed, line width=1.5pt] (2,1.58)--(2.37,2) node[right]{$\g$};
				
				\draw[orange, line width=1.5pt] (1,0)--(2,1.15);
				\draw[orange, dashed, line width=1.5pt] (-0.49,-1.71)--(1,0);
				\draw[orange, dashed, line width=1.5pt] (2,1.15)--(2.74,2) node[right]{$\g_s$};
				
				\draw (0,0) node{$\bullet$} node[above left]{$A$};
				\draw (1,0) node{$\bullet$} node[above left]{$B$};
				\draw (2,0) node{$\bullet$} node[above right]{$C$};
				\draw (0,1) node{$\bullet$} node[above left]{$D$};
				\draw (1,1) node{$\bullet$} node[above left]{$E$};
				\draw (2,1) node{$\bullet$} node[above right]{$F$};
				\draw (0,2) node{$\bullet$} node[above]{$G$};
				\draw (1,2) node{$\bullet$} node[above]{$H$};
				\draw (2,2) node{$\bullet$} node[above]{$I$};
				
				\draw[blue] (-1.06,-1.22) node{$\bullet$} node[left]{$A_1$};
				\draw[blue] (-0.49,-1.71) node{$\bullet$} node[left]{$B_1$};
				\draw[blue] (0.08,-2.21) node{$\bullet$} node[above]{$C_1$};
				\draw[blue] (-1.56,-0.79) node{$\bullet$} node[above]{$D_1$};
				\draw[blue] (-0.99,-1.28) node{$\bullet$} node[above right]{$E_1$};
				\draw[blue] (-0.42,-1.78) node{$\bullet$} node[below]{$F_1$};
				\draw[blue] (-2.05,-0.36) node{$\bullet$} node[above]{$G_1$};
				\draw[blue] (-1.48,-0.85) node{$\bullet$} node[below]{$H_1$};
				\draw[blue] (-0.91,-1.35) node{$\bullet$} node[below]{$I_1$};

			\end{tikzpicture}
			
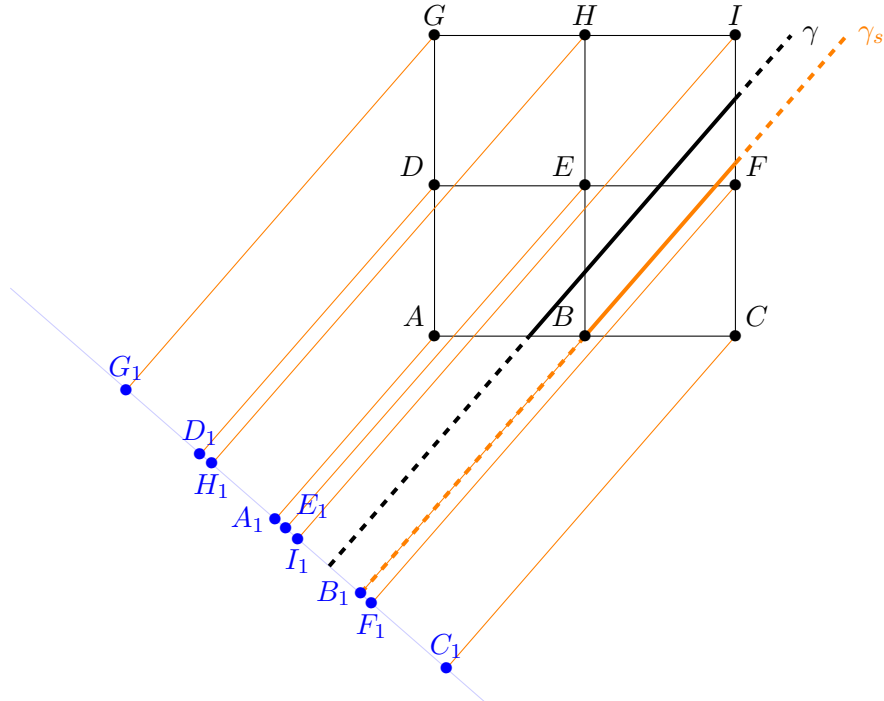
\captionof{figure}{Orthogonal projection of $M$ onto $\g$ in dimension $d=2$.}
		\end{center}
		
		\item In dimension $d=3$, when performing the same projection (projection of the edges of $\cG^n$ orthogonally onto $\g$), not all corners of the projection are "significant", meaning they do not all correspond to corners as defined.
		In fact, two edges can intersect on the projection without actually intersecting in reality. This leads to spurious corners in the projection.
		The problem is that these spurious corners depend on $\g$ (because they depend on the projection). Therefore, reducing the problem to one of these spurious corners through translations is not sufficient: Indeed, there are more spurious corners related to projections than there are geodesics in $\Gamma$, so it wouldn't even simplify the study of $\ell^T\left(\omega_{\varepsilon}^n\right)$. However, encountering a spurious corner in the projection indicates leaving one of the cells crossed by $\g$ and entering a new one. This causes $s\longmapsto m^T_{\g_s}\left(\omega\right)$ to lose its affine character.
		
		\begin{center}
			\begin{tikzpicture}[scale=2.5]
				
				\draw (1,0.23,1.52)--(1.49,0.18,0.73) node[right]{$\g$};
				\draw[dashed] (1,0.23,1.52)--(0.7,0.26,2);
				\draw (0.7,0.26,2)--(-0.12,0.34,3.31);
				
				\draw (0.09,0.09,3.46) node{$\bullet$} node[below]{$I_1$};
				\draw (0.12,1.09,3.41) node{$\bullet$} node[above]{$J_1$};
				\draw (-0.63,0.07,3.01) node{$\bullet$} node[below]{$L_1$};
				\draw (-0.6,1.06,2.96) node{$\bullet$} node[below left]{$K_1$};
				\draw (-0.33,1.14,3.13) node{$\bullet$} node[above]{$F_1$};
				\draw (-0.35,0.14,3.18) node{$\bullet$} node[above right]{$G_1$};
				\draw (-0.8,0.19,2.9) node{$\bullet$} node[above left]{$B_1$};
				\draw (-1.08,0.11,2.73) node{$\bullet$} node[below left]{$H_1$};
				\draw (-1.52,0.16,2.45) node{$\bullet$} node[above left]{$A_1$};
				\draw (-1.49,1.15,2.4) node{$\bullet$} node[above]{$D_1$};
				\draw (-1.05,1.11,2.68) node{$\bullet$} node[below left]{$E_1$};
				\draw (-0.77,1.18,2.85) node{$\bullet$} node[above]{$C_1$};
				\draw (-0.38,-0.86,3.22) node{$\bullet$} node[below]{$O_1$};
				\draw (-0.83,-0.81,2.94) node{$\bullet$} node[above right]{$M_1$};
				\draw (-1.1,-0.89,2.78) node{$\bullet$} node[below]{$P_1$};
				\draw (-1.55,-0.84,2.5) node{$\bullet$} node[below]{$N_1$};

				\draw[dashed] (0,0,0)--(1,0,0);
				\draw (1,0,0)--(1,1,0);
				\draw (1,1,0)--(0,1,0);
				\draw[dashed] (0,0,0)--(0,1,0);
				\draw[dashed] (0,0,1)--(1,0,1);
				\draw (1,0,1)--(1,1,1);
				\draw (1,1,1)--(0,1,1);
				\draw[dashed] (0,0,1)--(0,1,1);
				\draw[dashed] (0,0,0) -- (0,0,1); 
				\draw (1,0,0) -- (1,0,1); 
				\draw (1,1,0) -- (1,1,1); 
				\draw (0,1,0) -- (0,1,1); 
				\draw (0,0,2)--(1,0,2)--(1,1,2)--(0,1,2)--cycle;
				\draw[dashed] (0,0,2) -- (0,0,1); 
				\draw (1,0,2) -- (1,0,1); 
				\draw (1,1,2) -- (1,1,1); 
				\draw (0,1,2) -- (0,1,1); 
				\draw[dashed] (0,-1,0)--(1,-1,0); 
				\draw (1,-1,0)--(1,-1,1); 
				\draw (1,-1,1)--(0,-1,1); 
				\draw (0,-1,0)--(0,-1,1); 
				\draw[dashed] (0,0,0) -- (0,-1,0); 
				\draw (1,0,0) -- (1,-1,0); 
				\draw (1,0,1) -- (1,-1,1); 
				\draw[dashed] (0,0,1) -- (0,-0.387,1); 
				\draw (0,-0.387,1) -- (0,-1,1); 
				\draw (0,0,0) node{$\bullet$} node[above right]{$A$};
				\draw (1,0,0) node{$\bullet$} node[above right]{$B$};
				\draw (1,1,0) node{$\bullet$} node[above]{$C$};
				\draw (0,1,0) node{$\bullet$} node[above]{$D$};
				\draw (0,1,1) node{$\bullet$} node[above]{$E$};
				\draw (1,1,1) node{$\bullet$} node[above]{$F$};
				\draw (1,0,1) node{$\bullet$} node[below right]{$G$};
				\draw (0,0,1) node{$\bullet$} node[above left]{$H$};
				\draw (1,0,2) node{$\bullet$} node[below right]{$I$};
				\draw (1,1,2) node{$\bullet$} node[above]{$J$};
				\draw (0,1,2) node{$\bullet$} node[above]{$K$};
				\draw (0,0,2) node{$\bullet$} node[below]{$L$};
				\draw (1,-1,0) node{$\bullet$} node[above right]{$M$};
				\draw (0,-1,0) node{$\bullet$} node[above right]{$N$};
				\draw (1,-1,1) node{$\bullet$} node[below]{$O$};
				\draw (0,-1,1) node{$\bullet$} node[above left]{$P$};
				
				\draw[orange] (1,-1,1)--(-0.38,-0.86,3.22);
				\draw[orange] (0,-1,0)--(-1.55,-0.84,2.5);
				\draw[orange] (0,-1,1)--(-1.1,-0.89,2.78);
				\draw[orange] (1,-1,0)--(-0.83,-0.81,2.94);
				
				\draw[orange] (0,0,2)--(-0.63,0.07,3.01);
				\draw[orange] (0,1,2)--(-0.6,1.06,2.96);
				\draw[orange] (1,1,2)--(0.12,1.09,3.41);
				\draw[orange] (1,0,2)--(0.09,0.09,3.46);
				
				\draw[orange] (0,0,1)--(-1.08,0.11,2.73);
				\draw[orange] (1,0,1)--(-0.35,0.14,3.18);
				\draw[orange] (1,1,1)--(-0.33,1.14,3.13);
				\draw[orange] (0,1,1)--(-1.05,1.11,2.68);
				
				\draw[orange] (0,0,0)--(-1.52,0.16,2.45);
				\draw[orange] (1,0,0)--(-0.8,0.19,2.9);
				
				\draw[orange] (0,1,0)--(-1.49,1.15,2.4);
				\draw[orange] (1,1,0)--(-0.77,1.18,2.85);
				
				\draw[blue] (-1.49,1.15,2.4)--(-0.77,1.18,2.85);
				\draw[blue] (-1.49,1.15,2.4)--(-1.05,1.11,2.68);
				\draw[blue] (-0.77,1.18,2.85)--(-0.33,1.14,3.13);
				\draw[blue] (-1.49,1.15,2.4)--(-1.52,0.16,2.45);
				\draw[blue] (-0.77,1.18,2.85)--(-0.8,0.19,2.9);
				
				\draw[blue] (-1.52,0.16,2.45)--(-0.8,0.19,2.9);
				\draw[blue] (-1.52,0.16,2.45)--(-1.08,0.11,2.73);
				\draw[blue] (-0.8,0.19,2.9)--(-0.35,0.14,3.18);
				
				\draw[blue] (-1.08,0.11,2.73)--(-0.35,0.14,3.18)--(-0.33,1.14,3.13)--(-1.05,1.11,2.68)--cycle;
				
				\draw[blue] (-0.63,0.07,3.01)--(-0.6,1.06,2.96)--(0.12,1.09,3.41)--(0.09,0.09,3.46)--cycle;
				
				\draw[blue] (-1.05,1.11,2.68)--(-0.6,1.06,2.96);
				\draw[blue] (-0.33,1.14,3.13)--(0.12,1.09,3.41);
				\draw[blue] (-0.35,0.14,3.18)--(0.09,0.09,3.46);
				\draw[blue] (-1.08,0.11,2.73)--(-0.63,0.07,3.01);
				
				\draw[blue] (-1.55,-0.84,2.5)--(-0.83,-0.81,2.94)--(-0.38,-0.86,3.22)--(-1.1,-0.89,2.78)--cycle;
				
				\draw[blue] (-1.08,0.11,2.73)--(-1.1,-0.89,2.78);
				\draw[blue] (-0.35,0.14,3.18)--(-0.38,-0.86,3.22);
				\draw[blue] (-1.52,0.16,2.45)--(-1.55,-0.84,2.5);
				\draw[blue] (-0.8,0.19,2.9)--(-0.83,-0.81,2.94);
				
				\fill[color=blue][opacity=0.2] (-1.62,-0.98,2.46)--(0.25,-0.96,3.62)--(0.3,1.24,3.51)--(-1.57,1.22,2.35)--cycle;
				
				\begin{scope}[xshift=2cm,yshift=-0.5cm,xscale=-1,rotate=90]
					\draw[blue] (0,0.53) node{$\bullet$} node[below left]{$H_1$};
					\draw[blue] (0,0) node{$\bullet$} node[above left]{$A_1$};
					\draw[blue] (1,-0.08) node{$\bullet$} node[above]{$D_1$};
					\draw[blue] (1,0.45) node{$\bullet$} node[below left]{$E_1$};
					\draw[blue] (1.1,0.76) node{$\bullet$} node[above]{$C_1$};
					\draw[blue] (1,1) node{$\bullet$} node[below left]{$K_1$};
					\draw[blue] (1.1,1.29) node{$\bullet$} node[above]{$F_1$};
					\draw[blue] (1.1,1.82) node{$\bullet$} node[above]{$J_1$};
					\draw[blue] (0.1,1.91) node{$\bullet$} node[below]{$I_1$};
					\draw[blue] (0.1,1.38) node{$\bullet$} node[above right]{$G_1$};
					\draw[blue] (0,1.06) node{$\bullet$} node[below]{$L_1$};
					\draw[blue] (0.1,0.85) node{$\bullet$} node[above left]{$B_1$};
					\draw[blue] (-0.89,1.47) node{$\bullet$} node[below]{$O_1$};
					\draw[blue] (-0.89,0.93) node{$\bullet$} node[above right]{$M_1$};
					\draw[blue] (-1,0.6) node{$\bullet$} node[below]{$P_1$};
					\draw[blue] (-1,0.09) node{$\bullet$} node[below]{$N_1$};
					
					\draw (0,0)--(1,-0.08) ;
					\draw (0,0.53)--(1,0.45) ;
					\draw (0.1,0.85)--(1.1,0.76) ;
					\draw (0,1.06)--(1,1) ;
					\draw (0.1,1.38)--(1.1,1.29) ;
					\draw (0.1,1.91)--(1.1,1.82) ;
					\draw (-0.89,1.47)--(0.1,1.38) ;
					\draw (-0.89,0.93)--(0.1,0.85) ;
					\draw (-1,0.6)--(0,0.53) ;
					\draw (-1,0.09)--(0,0) ;
					\draw (1.1,1.82)--(1,1) ;
					\draw (1.1,1.29)--(1,0.45) ;
					\draw (1.1,0.76)--(1,-0.08) ;
					\draw (0.1,0.85)--(0,0) ;
					\draw (0.1,1.38)--(0,0.53) ;
					\draw (0.1,1.91)--(0,1.06) ;
					\draw (-0.89,1.47)--(-0.89,0.93) ;
					\draw (-1,0.6)--(-1,0.09) ;
					\draw (-0.89,1.47)--(-1,0.6);
					\draw (-0.89,0.93)--(-1,0.09);
					\draw (1.1,1.82)--(1.1,1.29);
					\draw (1.1,1.29)--(1.1,0.76);
					\draw (1,1)--(1,0.45);
					\draw (1,0.45)--(1,-0.08);
					\draw (0,0.53)--(0,0);
					\draw (0,1.06)--(0,0.53);
					\draw (0.1,1.38)--(0.1,0.85);
					\draw (0.1,1.91)--(0.1,1.38);
					
					\draw[orange] (0.04,1.385) node{$\bullet$};
					\draw[orange] (0.06,1.06) node{$\bullet$};
					\draw[orange] (0.1,1.05) node{$\bullet$};
					\draw[orange] (0.04,0.85) node{$\bullet$};
					\draw[orange] (0,0.86) node{$\bullet$};
					\draw[orange] (0.06,0.53) node{$\bullet$};
					\draw[orange] (1.04,1.3) node{$\bullet$};
					\draw[orange] (1,0.78) node{$\bullet$};
					\draw[orange] (1.04,0.77) node{$\bullet$};
					\draw[orange] (-0.93,0.595) node{$\bullet$};
				\end{scope}
			\end{tikzpicture}
			
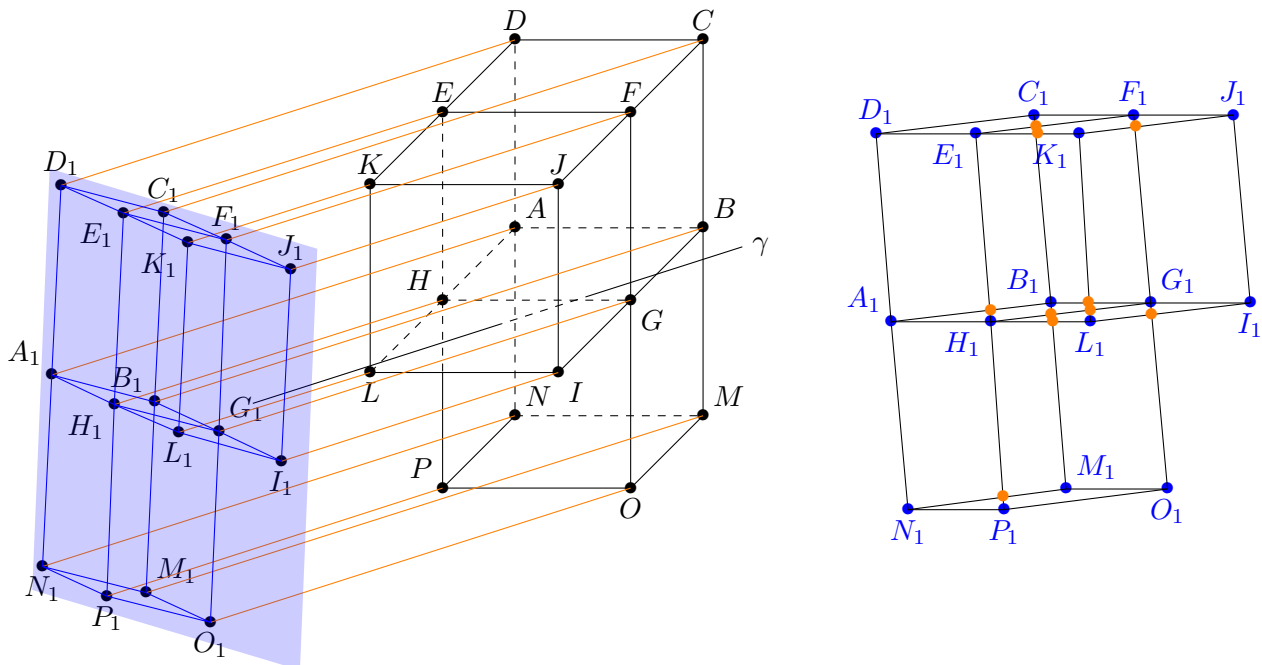
\captionof{figure}{Projection of $M$ orthogonally onto $\g$ in dimension $d=3$. In blue, the "significant" corners; in orange, the "spurious" corners.}
		\end{center}
	\end{itemize}
	
	Thus, we already encounter difficulties in dimension $d=3$ in trying to reduce, starting from an arbitrary geodesic in $M$, to a geodesic passing through a corner.
	
	The idea to extend the proof to higher dimensions has therefore been to group geodesics according to the different cells they encounter during their trajectory, rather than applying transformations to them.

	\section{Preliminaries}
	
	In this section, we present some results that allow us to restrict the study to $T \in \left(0,1\right)$ (Lemma \ref{restT}) and geodesics that are not completely contained in a hyperplane of the grid $\cG^n$ (Lemma \ref{restg}). This subset of $\Gamma$ is where $\g\longmapsto\m$ is continuous (Lemma \ref{continuite}).
	
	\begin{lemme}\label{restT}
		If $(ii)$ is true for every $T\in \left(0,1\right)$, then $(ii)$ is true for every $T>0$.
	\end{lemme}
	
	\begin{proof}
		Let $T>1$ and $m\in\N^*$ such that $T^\prime=\frac{T}{m}\in ]0,1[$. Let's show that $\ell^T\left(\omega_\varepsilon^n\right)\geqslant \ell^{T^\prime}\left(\omega_\varepsilon^n\right)$.
		Let $\rho>0$ and let $\g\in\Gamma$ a geodesic such that $\ell^T\left(\omega_\varepsilon^n\right)+\rho\geqslant \m$. We define, for every $k\in\llbracket 0, m-1\rrbracket$,  $\g_k\left(\cdot\right)=\g\left(kT^\prime+ \cdot \right)$. We then obtain:
		\begin{align*}
			\ell^T\left(\omega_\varepsilon^n\right)+\rho &\geqslant \m = \frac{1}{T}\int_0^T\chi_{\omega_\varepsilon^n}\left(\g\left(t\right)\right) \text{d}t = \frac{1}{T} \sum\limits_{k=0}^{m-1} \int_{kT^\prime}^{\left(k+1\right)T^\prime}\chi_{\omega_\varepsilon^n}\left(\g\left(t\right)\right) \text{d}t\\
			&= \frac{1}{T} \sum\limits_{k=0}^{m-1} \int_0^{T^\prime}\chi_{\omega_\varepsilon^n}\left(\g\left(kT^\prime+t\right)\right) \text{d}t = \frac{1}{T} \sum\limits_{k=0}^{m-1} \int_0^{T^\prime}\chi_{\omega_\varepsilon^n}\left(\g_k\left(t\right)\right) \text{d}t = \frac{1}{T} \sum\limits_{k=0}^{m-1} T^\prime m^{T^\prime}_{\g_k}\left(\omega_\varepsilon^n\right) \\
			&\geqslant \frac{T^\prime}{T} \sum\limits_{k=0}^{m-1} \ell^{T^\prime}\left(\omega_\varepsilon^n\right) = \ell^{T^\prime}\left(\omega_\varepsilon^n\right).
		\end{align*}
		By letting $\rho$ tend to $0$, we obtain $\ell^T\left(\omega_\varepsilon^n\right)\geqslant \ell^{T^\prime}\left(\omega_\varepsilon^n\right)$. Hence, we obtain the result.
	\end{proof}
	
	We now assume that $T \in ]0,1[$. Therefore, any geodesic $\gamma \in \Gamma$ crosses at most once the cell $c_{i_1,\dots,i_d}$ over the interval $[0,T]$.
	
	We denote by $\Gamma_H\subset\Gamma$ the set of geodesics that are entirely contained within a hyperplane of the grid $\cG^n$.
	It is clear that $\g\longmapsto\m$ is not continuous on the entire $\Gamma$, as shown by the following counterexample:
	
	\begin{center}
		\begin{tikzpicture}[scale=2]
			
			\fill[blue,opacity=0.2] (-1,1,2) node[left]{hyperplane $h$} --(1,1,2)--(1,1,0)--(-1,1,0)--cycle;
			
			\draw[blue,opacity=0.5] (-1,1,2) node[left]{hyperplane $h$};
			
			\fill[orange,opacity=0.2] (0,1,1)--(0,1,2)--(0,2,2)--(0,2,1)--cycle;
			\fill[orange,opacity=0.2] (1,1,1)--(1,1,2)--(1,2,2)--(1,2,1)--cycle;
			\fill[orange,opacity=0.2] (0,1,1)--(0,1,2)--(1,1,2)--(1,1,1)--cycle;
			\fill[orange,opacity=0.2] (0,2,1)--(0,2,2)--(1,2,2)--(1,2,1)--cycle;
			\fill[orange,opacity=0.2] (0,1,1)--(0,2,1)--(1,2,1)--(1,1,1)--cycle;
			\fill[orange,opacity=0.2] (0,1,2)--(0,2,2)--(1,2,2)--(1,1,2)--cycle;
			
			\fill[orange,opacity=0.2] (0,1,0)--(0,1,1)--(0,2,1)--(0,2,0)--cycle;
			\fill[orange,opacity=0.2] (1,1,0)--(1,1,1)--(1,2,1)--(1,2,0)--cycle;
			\fill[orange,opacity=0.2] (0,1,0)--(0,1,1)--(1,1,1)--(1,1,0)--cycle;
			\fill[orange,opacity=0.2] (0,2,0)--(0,2,1)--(1,2,1)--(1,2,0)--cycle;
			\fill[orange,opacity=0.2] (0,1,0)--(0,2,0)--(1,2,0)--(1,1,0)--cycle;
			
			\fill[orange,opacity=0.2] (0,0,0)--(0,0,1)--(0,1,1)--(0,1,0)--cycle;
			\fill[orange,opacity=0.2] (1,0,0)--(1,0,1)--(1,1,1)--(1,1,0)--cycle;
			\fill[orange,opacity=0.2] (0,0,0)--(0,0,1)--(1,0,1)--(1,0,0)--cycle;
			\fill[orange,opacity=0.2] (0,0,0)--(0,1,0)--(1,1,0)--(1,0,0)--cycle;
			\fill[orange,opacity=0.2] (0,0,1)--(0,1,1)--(1,1,1)--(1,0,1)--cycle;
			
			\draw[dashed] (-1,0,0)--(-1,1,0);
			\draw[dashed] (-1,1,0)--(-1,1,1);
			\draw[dashed] (-1,1,1)--(-1,0,1);
			\draw[dashed] (-1,0,0)--(-1,0,1);
			\draw[dashed] (-1,1,0)--(-1,2,0);
			\draw (-1,2,0)--(-1,2,1);
			\draw[dashed] (-1,2,1)--(-1,1,1);
			\draw[dashed] (-1,0,1)--(-1,0,2);
			\draw[dashed] (-1,1,1)--(-1,1,2);
			\draw (-1,0,2)--(-1,1,2);
			\draw (-1,1,2)--(-1,2,2);
			\draw (-1,2,2)--(-1,2,1);
			
			\draw[dashed] (0,0,0)--(-1,0,0);
			\draw[dashed] (0,1,0)--(-1,1,0);
			\draw[dashed] (0,1,1)--(-1,1,1);
			\draw (-1,2,0)--(0,2,0);
			\draw (-1,2,1)--(0,2,1);
			\draw[dashed] (0,0,1)--(-1,0,1);
			\draw (-1,1,2)--(0,1,2);
			\draw (0,2,2)--(-1,2,2);
			\draw (0,0,2)--(-1,0,2);
			
			\draw[dashed] (0,0,0)--(0,1,0);
			\draw[dashed] (0,1,0)--(0,1,1);
			\draw[dashed] (0,1,1)--(0,0,1);
			\draw[dashed] (0,0,0)--(0,0,1);
			\draw[dashed] (0,1,0)--(0,2,0);
			\draw (0,2,0)--(0,2,1);
			\draw[dashed] (0,2,1)--(0,1,1);
			\draw[dashed] (0,0,1)--(0,0,2);
			\draw[dashed] (0,1,1)--(0,1,2);
			\draw (0,0,2)--(0,1,2);
			\draw (0,1,2)--(0,2,2);
			\draw (0,2,2)--(0,2,1);
			
			\draw[dashed] (0,0,0)--(1,0,0);
			\draw[dashed] (0,1,0)--(1,1,0);
			\draw[dashed] (0,1,1)--(1,1,1);
			\draw (1,2,0)--(0,2,0);
			\draw (1,2,1)--(0,2,1);
			\draw[dashed] (0,0,1)--(1,0,1);
			\draw (1,1,2)--(0,1,2);
			\draw (0,2,2)--(1,2,2);
			\draw (0,0,2)--(1,0,2);
			
			\draw (1,0,0)--(1,1,0);
			\draw (1,1,0)--(1,1,1);
			\draw (1,1,1)--(1,0,1);
			\draw (1,0,0)--(1,0,1);
			\draw (1,1,0)--(1,2,0);
			\draw (1,2,0)--(1,2,1);
			\draw (1,2,1)--(1,1,1);
			\draw (1,0,1)--(1,0,2);
			\draw (1,1,1)--(1,1,2);
			\draw (1,0,2)--(1,1,2);
			\draw (1,1,2)--(1,2,2);
			\draw (1,2,2)--(1,2,1);
			
			\draw[orange, dashed,line width=1pt] (0.24,0.64,2)--(1,0.64,0.58);
			\draw[blue, dashed,line width=1pt] (0.24,1,2)--(1,1,0.58);
			\draw[orange, dashed,line width=1pt] (0.24,1.34,2)--(1,1.34,0.58);
			
			\draw[orange,line width=1pt] (1.31,0.64,0)--(1,0.64,0.58);
			\draw[blue,line width=1pt] (1.31,1,0)--(1,1,0.58);
			\draw[orange,line width=1pt] (1.31,1.34,0)--(1,1.34,0.58);
			
			\draw[orange,line width=1pt] (0.24,0.64,2)--(-1,0.64,4.32);
			\draw[blue,line width=1pt] (0.24,1,2)--(-1,1,4.32);
			\draw[orange,line width=1pt] (0.24,1.34,2)--(-1,1.34,4.32);
			
			\draw[orange,line width=1pt] (1.31,0.64,0) node[right]{$\g_k^1$};
			\draw[blue,line width=1pt] (1.31,1,0) node[right]{$\g$};
			\draw[orange,line width=1pt] (1.31,1.34,0) node[right]{$\g_k^2$};
			
		\end{tikzpicture}
		
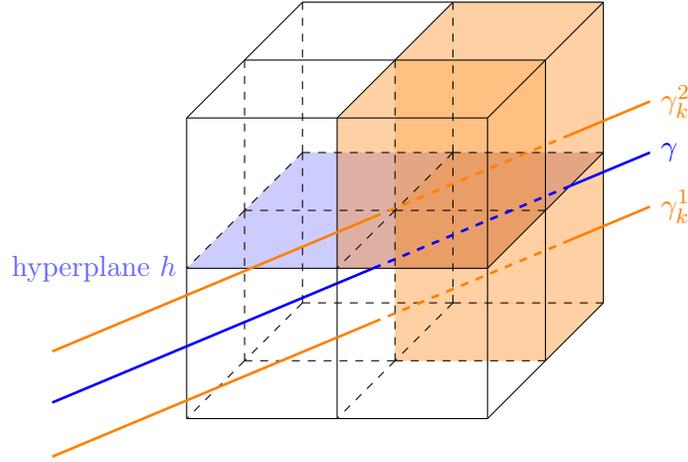
\captionof{figure}{$\g$ is contained within the blue hyperplane $h$. $\omega_\varepsilon^n$ is the union of the cubes colored in orange.}
	\end{center}
	Indeed, the sequence of geodesics $\left(\g_k^1\right)_k$, parallel to $\g$ and converging to $\g$ from below the hyperplane $h$, satisfies that for every $k\in\N$, $m_{\g_k^1}^T\left(\omega_\varepsilon^n\right)=C_1$, where $C_1>0$ is the time spent by $\g_k^1$ in the cube at the back bottom-right of $M$. The sequence of geodesics $\left(\g_k^2\right)_k$, parallel to $\g$ and converging to $\g$ from above the hyperplane $h$, satisfies that for every $k\in\N$, $m_{\g_k^2}^T\left(\omega_\varepsilon^n\right)=C_1+C_2$, where $C_2>0$ is the time spent by $\g_k^2$ in the cube at the front top-right of $M$ ($C_1$ is also equal to the time spent by $\g_k^2$ in the cube at the back top-right of $M$). Thus, $\g_k^1$ and $\g_k^2$ have the same limit $\g$, but not $m_{\g_k^1}^T\left(\omega_\varepsilon^n\right)$ and $m_{\g_k^2}^T\left(\omega_\varepsilon^n\right)$.

	However, $\g\longmapsto\m$ is continuous on the subset $\Gamma\backslash\Gamma_H$.
	
	\begin{lemme}\label{continuite}
		$\g\longmapsto\m$ is continuous on $\Gamma\backslash\Gamma_H$.
	\end{lemme}
	
	\begin{proof}
		Let $\g\in \Gamma\backslash\Gamma_H$. Let $\left(\g_k\right)_{k\in\N}$ be a sequence of geodesics converging to $\g$. We want to show that $\left(\mk\right)_{k\in\N}$ converges to $\m$.
		
		We write $\co=\sum\limits_{i_1, \dots, i_d=1}^n X_{i_1, \dots, i_d}\cc$ 
		and $\mk=\sum\limits_{i_1, \dots, i_d=1}^n X_{i_1, \dots, i_d} \mkc$. Let's show that for any $\left(i_1, \dots, i_d\right)\in \llbracket 1, n \rrbracket^d$, the sequence $\left(\mkc\right)_{k\in\N}$ converges to $\mc$. Consider $\left(i_1, \dots, i_d\right)\in \llbracket 1, n \rrbracket^d$, then $\left(\ccgk\right)_{k\in\N}$ converges almost everywhere to $\ccg$. Indeed, if $\g\left(t\right)\in \icij$, then for sufficiently large $k$, we have $\g_k\left(t\right)\in \icij$, and $\ccgkt=\ccgt=1$. Similarly, if $\g\left(t\right)\notin \cij$, then for sufficiently large $k$, we have $\g_k\left(t\right)\notin \cij$, and $\ccgkt=\ccgt=0$. Moreover, the set $\left\{t\in\left[0,T\right], \g\left(t\right)\in \partial \cij \right\}$ (where $ \partial \cij $ denotes the boundary of $ \cij $) is finite since $\g\in \Gamma\backslash\Gamma_H$. As $\mkc$ $=\frac{1}{T}\int_0^T \ccgkt \text{d}t$, we can conclude by dominated convergence.
	\end{proof}
	
	\begin{lemme}\label{restg}
		We have $\inf\limits_{\g \in \Gamma} \m=\inf\limits_{\g \in \Gamma \backslash \Gamma_H} \m$.
	\end{lemme}
	
	\begin{proof}
		Let $\g\in \Gamma_H$. Then, there exist $i\in \llbracket 1, d \rrbracket$ and $j\in \N^*$ such that $\g$  is entirely contained within the $j$-th hyperplane of type $i$, denoted as $f^i_j$. Let $c_1, \dots, c_N$ be the cells encountered by $\g$ on one side of the hyperplane $f^i_j$, and $c_1^\prime, \dots, c_N^\prime$ be the cells encountered on the other side of the hyperplane $f^i_j$ (see Figure \ref{restgg} left). We then have the following equality: $\m=\sum\limits_{l=1}^{N}\frac{t_l\left(\g\right)}{T}Y_l$ where, for all $l\in \llbracket 1, N \rrbracket$, $Y_l$  is the random variable defined as $Y_l=1$ if $c_l\in \omega_\varepsilon^n$ or $c_l^\prime\in \omega_\varepsilon^n$, and $Y_l=0$ otherwise, and $t_l\left(\g\right)$ denotes the time spent by $\g$ in the cell $c_l$ (which is equal to the time spent in the corresponding cell $c_l^\prime$). 
		
		In the example of Figure \ref{restgg}, we have $Y_1=Y_2=Y_3=Y_4=Y_5=Y_6=Y_7=Y_{11}=Y_{13}=Y_{14}=1$ and $Y_8=Y_9=Y_{10}=Y_{15}=0$.
		
		Now, let's consider $\g_t=\tau_{te_i}\circ \g$ for $t\in \left] 0,\frac{1}{n}\right[$, where $\tau_{te_i}$ denotes the translation by the vector $te_i$ (see Figure \ref{restggt}). We have $\mt\in\G$ (where $\G$ is the subset of $\Gamma$ that contains only geodesics propagating on $M$ without encountering any edges of the grid $\cG^n$), and $\mt=\sum\limits_{l=1}^{N}\frac{t_l\left(\g\right)}{T}X_l$ (where, for all $l\in \llbracket 1, N \rrbracket$, $X_l=1$ if $c_l\in \omega_\varepsilon^n$ and $Y_l=0$ otherwise). 
		
		In the example of Figure \ref{restggt}, we have $X_2=X_3=X_4=X_5=X_{11}=X_{13}=X_{14}=1$ and $X_1=X_6=X_7=X_8=X_9=X_{10}=X_{12}=X_{15}=0$. 
		
		Thus, we obtain $\inf\limits_{\g \in \Gamma} \m\geqslant\inf\limits_{\g \in \G} \m$, and since $\G\subset \Gamma$, we have $\inf\limits_{\g \in \Gamma} \m=\inf\limits_{\g \in \G} \m$.
	\end{proof}		

		\begin{center}
			\begin{tikzpicture}[scale=0.5]
				\draw (0.5,2) node[above]{$c_1$};
				\draw (1.5,2) node[above]{$c_2$};
				\draw (2.5,2) node[above]{$c_3$};
				\draw (3.5,2) node[above]{$c_4$};
				\draw (14.5,2) node[above]{$c_{15}$};
				\draw (0.5,0) node[below]{$c_1^{\prime}$};
				\draw (1.5,0) node[below]{$c_2^{\prime}$};
				\draw (2.5,0) node[below]{$c_3^{\prime}$};
				\draw (3.5,0) node[below]{$c_4^{\prime}$};
				\draw (14.5,0) node[below]{$c_{15}^{\prime}$};
				\fill (0,0) rectangle (1,1);
				\fill (3,0) rectangle (4,1);
				\fill (5,0) rectangle (6,1);
				\fill (6,0) rectangle (7,1);
				\fill (12,0) rectangle (13,1);
				\fill (1,1) rectangle (2,2);
				\fill (2,1) rectangle (3,2);
				\fill (3,1) rectangle (4,2);
				\fill (4,1) rectangle (5,2);
				\fill (10,1) rectangle (11,2);
				\fill (12,1) rectangle (13,2);
				\fill (13,1) rectangle (14,2);
				\draw (0,0) grid (15,2);
				\draw [orange][line width=2pt] (-0.5,1) -- (15.5,1) node[above]{$\g$};
				\fill [color=orange!50](17,0) rectangle (18,1);
				\fill [color=orange!50](20,0) rectangle (21,1);
				\fill [color=orange!50](22,0) rectangle (23,1);
				\fill [color=orange!50](23,0) rectangle (24,1);
				\fill [color=orange!50](29,0) rectangle (30,1);
				\fill [color=orange!50](18,1) rectangle (19,2);
				\fill [color=orange!50](19,1) rectangle (20,2);
				\fill [color=orange!50](20,1) rectangle (21,2);
				\fill [color=orange!50](21,1) rectangle (22,2);
				\fill [color=orange!50](27,1) rectangle (28,2);
				\fill [color=orange!50](29,1) rectangle (30,2);
				\fill [color=orange!50](30,1) rectangle (31,2);
				\draw (17,0) grid (32,2);
				\draw [orange][line width=2pt] (16.5,1) -- (32.5,1) node[above]{$\g$};
			\end{tikzpicture}
			
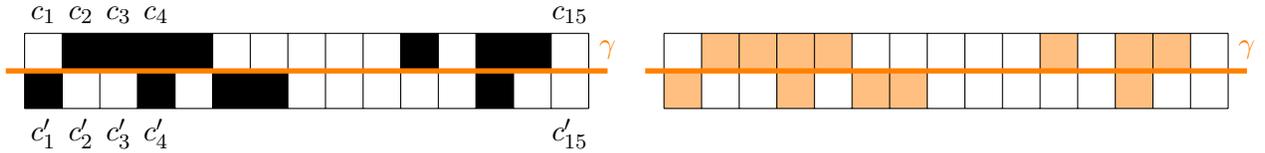
\captionof{figure}{Geodesic $\g\in\Gamma_H$. The orange cells represent the cells of $\omega_\varepsilon^n$ that $\g$ encounters. In this example, we have $Y_1=Y_2=Y_3=Y_4=Y_5=Y_6=Y_7=Y_{11}=Y_{13}=Y_{14}=1$ and $Y_8=Y_9=Y_{10}=Y_{15}=0$.}\label{restgg}
			
		\end{center}
		
		\begin{center}
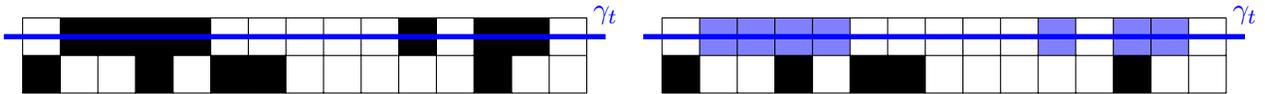

			\begin{tikzpicture}[scale=0.5]
				\fill (0,0) rectangle (1,1);
				\fill (3,0) rectangle (4,1);
				\fill (5,0) rectangle (6,1);
				\fill (6,0) rectangle (7,1);
				\fill (12,0) rectangle (13,1);
				\fill (1,1) rectangle (2,2);
				\fill (2,1) rectangle (3,2);
				\fill (3,1) rectangle (4,2);
				\fill (4,1) rectangle (5,2);
				\fill (10,1) rectangle (11,2);
				\fill (12,1) rectangle (13,2);
				\fill (13,1) rectangle (14,2);
				\draw (0,0) grid (15,2);
				\draw [blue][line width=2pt] (-0.5,1.5) -- (15.5,1.5) node[above]{$\g_t$};
				\fill (17,0) rectangle (18,1);
				\fill (20,0) rectangle (21,1);
				\fill (22,0) rectangle (23,1);
				\fill (23,0) rectangle (24,1);
				\fill (29,0) rectangle (30,1);
				\fill [color=blue!50](18,1) rectangle (19,2);
				\fill [color=blue!50](19,1) rectangle (20,2);
				\fill [color=blue!50](20,1) rectangle (21,2);
				\fill [color=blue!50](21,1) rectangle (22,2);
				\fill [color=blue!50](27,1) rectangle (28,2);
				\fill [color=blue!50](29,1) rectangle (30,2);
				\fill [color=blue!50](30,1) rectangle (31,2);
				\draw (17,0) grid (32,2);
				\draw [blue][line width=2pt] (16.5,1.5) -- (32.5,1.5) node[above]{$\g_t$};
			\end{tikzpicture}
			\captionof{figure}{Geodesic $\g_t\in\Gamma\backslash\Gamma_H$, translated from $\g$. The blue cells represent the cells of $\omega_\varepsilon^n$ that $\g_t$ encounters. In this example, we have $X_2=X_3=X_4=X_5=X_{11}=X_{13}=X_{14}=1$ and $X_1=X_6=X_7=X_8=X_9=X_{10}=X_{12}=X_{15}=0$.} \label{restggt}
		\end{center}
		

	Thus, as a consequence of Lemma \ref{restg} it suffices to consider the geodesics from $\Gamma\backslash\Gamma_H$.

	\section{Proof of \ref{thii}}\label{secpr}
	
	We recall that by $\G$\label{gammachapeau}, we denote the subset of $\Gamma$ that contains only geodesics propagating on $M$ without encountering any edges of the grid $\cG^n$.
	
	Let $\cP_d$ be the set of sequences of cells $\left(c_1, \dots, c_k\right)$ such that there exists a geodesic $\g\in\G$ that traverses the cells $c_1, \dots, c_k$ in this order, and only those cells.
	
	We divide $\G$ into $\#\cP_d$ subsets $\G_p\neq \emptyset$, $p=\left(c_1, \dots, c_k\right)\in\cP_d$, as follows: for all $\g\in\G_p$, $\g$ traverses the cells $c_1, \dots, c_k$ in this order, and only those cells. 
	Thus, two elements belonging to the same $\G_p$ encounter the same cells of the grid $\cG^n$ in the same order. For any $p\in\cP_d$, we consider $\g_p\in \G_p$ as a representative of the class $\G_p$.
	
	
	\subsection{Number of classes}
	
	The objective of this subsection is to evaluate the number of classes $\G_p$, $p\in\cP_d$, and more particularly to show that this number is polynomial in $n$.

	\begin{lemme}\label{ppoly}
		$\#\cP_d$ is polynomial in $n$.
	\end{lemme}
	
	\begin{lemme}
		If Lemma \ref{ppoly} is true for $d=2$, then it is true for any $d\geqslant 2$.
	\end{lemme}

	\begin{proof}
		For any $i, j \in \llbracket 1,d\rrbracket$ such that $i\neq j$, we denote $\cP_d\left(i,j\right)$ the set of ways to place hyperplanes of type $i$ with respect to those of type $j$ on the trajectory of a geodesic of $\G$. We assume that $\#\cP_2=\#\cP_d\left(2,1\right)$ is polynomial in $n$. Then, $\#\cP_d \leqslant \prod\limits_{i\neq j} \#\left(\cP_d\left(i,j\right)\right)$.
		Indeed, if we fix the order of hyperplanes of type $i$ with respect to those of type $j$, the order of hyperplanes of type $i$ with respect to those of type $k$, and the order of hyperplanes of type $j$ with respect to those of type $k$, then the sequence formed by hyperplanes of type $i$, $j$, and $k$ is determined.
		Then, $\#\cP_d \leqslant \left(\#\cP_d\left(2,1\right)\right)^{\binom{d}{2}}$ is polynomial in $n$.
	\end{proof}
	
We now consider the case where $d=2$.	
	
Let $(i,j)$ be a vertex of the grid $\cG_n$. We will denote $C(i,j)$ as the cell of the grid $\cG_n$ whose lower-left vertex is $(i,j)$ (see Figure \ref{notcel}).

\begin{center}
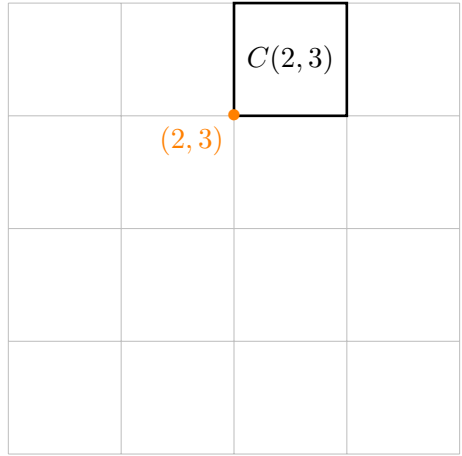

	\begin{tikzpicture}[scale=1.5]
		\draw[gray,opacity=0.5] (0,0) grid (4,4);
		\draw[line width=1pt] (2,3) rectangle (3,4);
		\draw (2.5,3.5) node{$C(2,3)$};
		\draw[orange] (2,3) node{$\bullet$} node[below left]{$(2,3)$};
	\end{tikzpicture}
	\captionof{figure}{Notation for the cells of $\cG^n$.}\label{notcel}
\end{center}

Consider two vertices $(i,j)$ and $(i',j')$ of the grid $\cG_n$. Without loss of generality, we assume that $i'>i$ and $j'>j$. We denote $c=C(i,j)$ and $c'=C(i',j')$.

\begin{center}
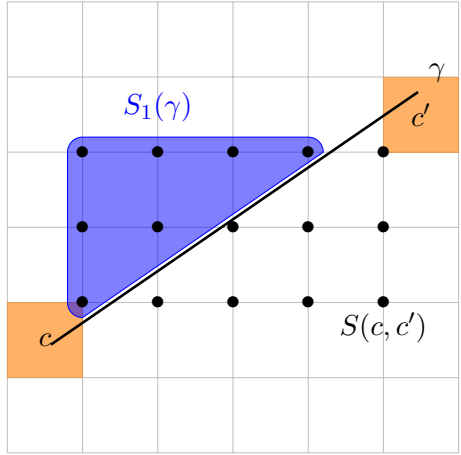

	\begin{tikzpicture}
		\draw[gray,opacity=0.5] (0,0) grid (6,6);
		\fill[orange, opacity=0.6] (0,1) rectangle (1,2);
		\fill[orange, opacity=0.6] (5,4) rectangle (6,5);
		\draw (0.5,1.5) node{$c$};
		\draw (5.5,4.5) node{$c'$};
		\draw[blue] (2,4.6) node{$S_1(\g)$};
		\draw[blue] (0.8,2)--(0.8,4) arc (180:90:0.2) -- (4,4.2) arc (90:0:0.2) -- (1,1.8) arc (270:180:0.2);
		\fill[blue, opacity=0.5] (0.8,2)--(0.8,4) arc (180:90:0.2) -- (4,4.2) arc (90:0:0.2) -- (1,1.8) arc (270:180:0.2);
		\draw[line width=1pt] (0.58,1.44)--(5.46,4.8) node[above right]{$\g$};
		\draw (1,2) node{$\bullet$};
		\draw (1,3) node{$\bullet$};
		\draw (1,4) node{$\bullet$};
		\draw (2,2) node{$\bullet$};
		\draw (2,3) node{$\bullet$};
		\draw (2,4) node{$\bullet$};
		\draw (3,2) node{$\bullet$};
		\draw (3,3) node{$\bullet$};
		\draw (3,4) node{$\bullet$};
		\draw (4,2) node{$\bullet$};
		\draw (4,3) node{$\bullet$};
		\draw (4,4) node{$\bullet$};
		\draw (5,2) node{$\bullet$} node[below]{$S(c,c')$};
		\draw (5,3) node{$\bullet$};
		\draw (5,4) node{$\bullet$};
	\end{tikzpicture}
	\captionof{figure}{Example of fixed starting and ending cells.}\label{deparr}
\end{center}

We denote
\begin{equation*}
	\cP(c,c')=\left\{p \in \cP \, \middle| \, \forall \g \in \G_p, \, \g(0)\in c \text{ and } \g(T)\in c'\right\},
\end{equation*} 
the set of classes of geodesics that start at position $c$ and end at position $c'$. We denote
\begin{equation*}
	\G(c,c')=\left\{ \g \in \G \, \middle| \, \exists \, p \in \cP(c,c') , \, \g \in \G_p \right\},
\end{equation*} 
the set of geodesics in $\G$ that start at position $c$ and end at position $c'$.

\begin{proposition}
	$\#\cP(c,c')$ is polynomial in $n$.
\end{proposition} 

\begin{proof}
	Counting the sequences of cells $p \in \cP$ taken by the geodesics of $\G$ amounts to counting the sequences of hyperplanes in $\cG^n$ crossed by them. Let $S(c,c')=\left\{(k,l)\in S \, \middle| \, i < k \leqslant i' \text{ and } j < l \leqslant j'\right\}$ be the set of vertices in the grid $\cG^n$ contained in the rectangle formed by the cells $c=C(i,j)$ and $c'=C(i',j')$.
	For any $\g\in\G$, let $t_k^V(\g)$ be the time at which $\g$ intersects the $k$-th vertical hyperplane (type 1 hyperplane) of the grid $\cG^n$, and let $t_l^H(\g)$ be the time at which $\g$ intersects the $l$-th horizontal hyperplane (type 2 hyperplane) of the grid $\cG^n$ for any $k\in \llbracket i+1, i' \rrbracket$ and $l\in \llbracket j+1, j' \rrbracket$.
	Define $S_1(\g)=\left\{(k,l)\in S(c,c') \, \middle| \, t_k^V(\g)<t_l^H(\g)\right\}$ as the set of vertices in $S(c,c')$ located in the upper-left triangle delimited by $\g$, and similarly, let $S_2(\g)=\left\{(k,l)\in S(c,c') \, \middle| \, t_k^V(\g)>t_l^H(\g)\right\}$ be the set of vertices in $S(c,c')$ located in the lower-right triangle delimited by $\g$.
	Thus, $S_1(\g)$ and $S_2(\g)$ form a partition of the set of vertices $S(c,c')$.
	
	\begin{claim}\label{cnsclass}
		Let $\g, \g^\prime \in\G(c,c')$. Then, $\g$ and $\g^\prime$ are in the same class if and only if they separate the vertices of the grid $\cG_n$ in the same way, i.e., if and only if $S_1(\g)=S_1(\g^\prime)$.
	\end{claim}
	
	\begin{proof}
		Let's assume that $\g$ and $\g^\prime$ are in the same class. Let's show that $S_1(\g)=S_1(\g^\prime)$.
		Take $(k,l)\in S_1(\g)$. We want to demonstrate that $(k,l)\in S_1(\g^\prime)$, i.e., $t_k^V(\g^\prime)<t_l^H(\g^\prime)$. 
		By assumption, $(k,l)\in S_1(\g)$, so $t_k^V(\g)<t_l^H(\g)$.
		Therefore, $\g$ intersects the $k$-th vertical hyperplane before encountering the $l$-th horizontal hyperplane of the $\cG^n$ grid. Now, $\g$ and $\g^\prime$ are in the same class by assumption, meaning these two geodesics intersect the same hyperplanes of $\cG^n$ and in the same order. Thus, $\g^\prime$ intersects the $k$-th vertical hyperplane before encountering the $l$-th horizontal hyperplane of the $\cG^n$ grid. Hence, $t_k^V(\g^\prime)<t_l^H(\g^\prime)$.
		Therefore, $S_1(\g)\subset S_1(\g^\prime)$, and then $S_1(\g)=S_1(\g^\prime)$ because $\g$ and $\g'$ play symmetric roles.
		
		Conversely, let's assume that $S_1(\g)=S_1(\g^\prime)$. Let's show that $\g$ and $\g^\prime$ are in the same class, meaning these two geodesics intersect the same hyperplanes of $\cG^n$ and in the same order.
		Since $\g, \g^\prime \in\G(c,c')$, it is obvious that $\g$ and $\g^\prime$ intersect the same hyperplanes of the $\cG^n$ grid, i.e., the $i$-th to $i'$-th vertical hyperplanes and the $j$-th to $j'$-th horizontal hyperplanes of the $\cG^n$ grid. It remains to show that $\g$ and $\g^\prime$ intersect them in the same order.
		Moreover, since $i'>i$, $\g$ and $\g'$ intersect the vertical hyperplanes from left to right. In other words, for any $k_1, k_2\in \llbracket i+1, i' \rrbracket$, if $k_1<k_2$, then $t^V_{k_1}(\g)<t^V_{k_2}(\g)$ and $t^V_{k_1}(\g')<t^V_{k_2}(\g')$.
		Similarly, since $j'>j$, $\g$ and $\g^\prime$ intersect the horizontal hyperplanes from bottom to top. That is, for any $l_1, l_2\in \llbracket j+1, j' \rrbracket$, if $l_1<l_2$, then $t^H_{l_1}(\g)<t^H_{l_2}(\g)$ and $t^H_{l_1}(\g')<t^H_{l_2}(\g')$.
		Take $k\in \llbracket i+1, i' \rrbracket$ and $l\in \llbracket j+1, j' \rrbracket$, and assume that $t^V_{k}(\g)<t^H_{l}(\g)$. Then, $(k,l)\in S_1(\g)$, by the definition of $S_1(\g)$. Now, $S_1(\g)=S_1(\g^\prime)$ by assumption, so $t^V_{k}(\g')<t^H_{l}(\g')$.
		Similarly, if $t^V_{k}(\g)>t^H_{l}(\g)$, then $t^V_{k}(\g')>t^H_{l}(\g')$, because $S_2(\g)=S_2(\g')$.
		Hence, $\g$ and $\g^\prime$ are in the same class.
	\end{proof}
	
	This proves that $\#\cP(c,c')$ is bounded by the number of ways to strictly separate $n^2$ points with a line in the plane.

	We conclude the proof by applying Lemma \ref{sepnpt}: the number of ways to strictly separate $n^2$ points with a line in the plane is polynomial in $n$, so $\#\cP(c,c')$ is polynomial in $n$.
\end{proof}

	\subsection{Geodesics from the same class}
	
	The objective of this subsection is to control the difference $\left|\m-\mp\right|$ for two geodesics $\g$ and $\g'$ in the same class $\G_p$. The goal is to approximate any $\m$, $\g\in\G_p$, by $\mpp$, where $\g_p$ is a fixed representative of $\G_p$.

	In the following, we fix $p\in\cP_d$.
	
	Let $N_p + 1 \in \N^*$ be the number of cells in the grid $\cG^n$ that are encountered by the geodesics in $\G_p$. We denote $c_0$ as the starting cell for the geodesics in $\G_p$, and for each $i\in \llbracket 1,N_p \rrbracket$, $c_i$, the $\left(i+1\right)$-th cell encountered by the geodesics in $\G_p$.
	
	For any $\gamma\in\G_p$, for each $i\in \llbracket 1,N_p \rrbracket$, we denote $\si$ as the time at which $\gamma$ intersects its $i$-th hyperplane. Additionally, we set $s_0\left(\g\right)=0$ and $s_{N_p+1}\left(\g\right)=T$. For any $\gamma\in\G_p$ and $i\in \llbracket 1,d \rrbracket$, we define $k_i + 1\in\llbracket 1,N_p \rrbracket$ as the number of hyperplanes of type $i$ that $\gamma$ intersects. This quantity depends only on $p$ and not on the choice of $\gamma$ in $\G_p$. We then denote, for each $j\in \llbracket 0,k_i \rrbracket$, $\tij$ as the time at which $\gamma$ intersects its $j+1$-th hyperplane of type $i$. It should be noted that for any $\g \in \G_p$ $\left\{\si \; \middle| \; i \in \llbracket 1, N_p \rrbracket\right\} = \left\{\tij \; \middle| \; i \in \llbracket 1, N_p \rrbracket\right\}$, and more precisely, there exists a bijection $\phi_p : \llbracket 1, N_p \rrbracket \longrightarrow \left\{\left(i, j\right) \; \middle| \; i \in \llbracket 1, d\rrbracket, \; j \in \llbracket 0, k_i \rrbracket \right\}$ such that for any $\g\in\G_p$, for any $i \in \llbracket 1, d \rrbracket$, $\si=t_{\phi_p\left(i\right)}\left(\g\right)$, see Figure \ref{defst}.
	
	\begin{center}
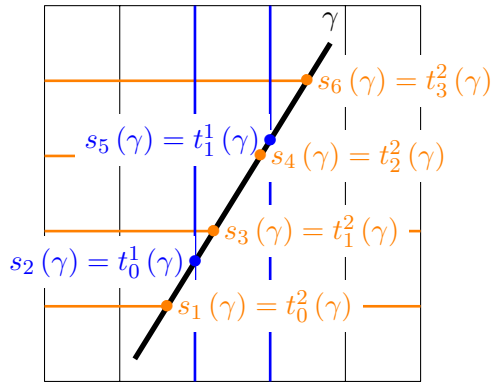


		\begin{tikzpicture}
			\draw (0,0) grid (5,5);
			\draw[orange][line width=1pt] (0,1) -- (5,1);
			\draw[orange][line width=1pt] (0,2) -- (5,2);
			\draw[orange][line width=1pt] (0,3) -- (5,3);
			\draw[orange][line width=1pt] (0,4) -- (5,4);
			\draw[blue][line width=1pt] (2,0) -- (2,5);
			\draw[blue][line width=1pt] (3,0) -- (3,5);
			\draw[orange] (1.63,1) node {$\bullet$} node[right][fill=white]{$\sf{1}=\tf{2}{0}$};
			\draw[orange] (2.25,2) node {$\bullet$} node[right][fill=white]{$\sf{3}=\tf{2}{1}$};
			\draw[orange] (2.87,3) node {$\bullet$} node[right][fill=white]{$\sf{4}=\tf{2}{2}$};
			\draw[orange] (3.49,4) node {$\bullet$} node[right][fill=white]{$\sf{6}=\tf{2}{3}$};
			\draw[blue] (2,1.59) node {$\bullet$}  node[left][fill=white]{$\sf{2}=\tf{1}{0}$};
			\draw[blue] (3,3.21) node {$\bullet$}  node[left][fill=white]{$\sf{5}=\tf{1}{1}$};
			\draw[line width=2pt] (1.2,0.3) -- (3.8,4.5) node[above]{$\g$};
			\draw[orange] (1.63,1) node {$\bullet$};
			\draw[orange] (2.25,2) node {$\bullet$};
			\draw[orange] (2.87,3) node {$\bullet$};
			\draw[orange] (3.49,4) node {$\bullet$};
			\draw[blue] (2,1.59) node {$\bullet$};
			\draw[blue] (3,3.21) node {$\bullet$};
		\end{tikzpicture}
		\captionof{figure}{Illustration of the notations on a geodesic $\g$ in dimension $d=2$.}\label{defst}
	\end{center}
	
	We define $N_p+1$ random variables $\left(X_i\right)_{i\in\llbracket 0,N_p \rrbracket}$ as follows: for any $i\in \llbracket 0,N_p \rrbracket$, $X_i=$ $\begin{cases} 1 ~~ \text{if} ~~ c_i\in\omega_\varepsilon^n\\
		0 ~~ \text{otherwise} ~~ \end{cases}$. The $\left(X_i\right)_{i\in\llbracket 0,N_p \rrbracket}$ are therefore independent random variables, identically distributed as Bernoulli random variables with parameter $\varepsilon$.
	
	Let $\g\in\G_p$, we can then write :
	
	\begin{align*}
		\m&=\frac{1}{T}\sum\limits_{i=0}^{N_p}\left(\sii-\si\right)X_i =\frac{1}{T}\sum\limits_{i=0}^{N_p}\sii X_i - \frac{1}{T}\sum\limits_{i=0}^{N_p}\si X_i\\
		&=\frac{1}{T}\sum\limits_{i=0}^{N_p}\sii X_i - \frac{1}{T}\sum\limits_{i=0}^{N_p-1}\sii X_{i+1} =\frac{1}{T}\sum\limits_{i=0}^{N_p-1}\sii \left(X_i-X_{i+1}\right)+X_{N_p+1} \\
		&=\frac{1}{T}\sum\limits_{i=1}^{d}\sum\limits_{j=0}^{k_i}\tij \left(X_{ij}-Y_{ij}\right)+X_{N_p+1},
	\end{align*}
	by grouping the hyperplanes by types and renaming the $\left(X_i, X_{i+1}\right)$ as $\left(X_{ij}, Y_{ij}\right)$.
	We observe that $X_{ij}$ and $Y_{ij}$ do not depend on $\g$ but only on the different encountered hyperplanes and their order, which is the same for every element of $\G_p$.
	
	Let $\g\in\G_p$, let $i\in \llbracket 1,d \rrbracket$. Between two consecutive hyperplanes of type $i$, $\g$ spends a constant time $\ai$ (Thales' theorem).

	\begin{center}
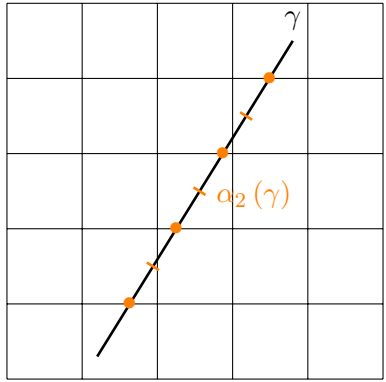


		\begin{tikzpicture}
			\draw (0,0) grid (5,5);
			\draw[line width=1pt] (1.2,0.3) -- (3.8,4.5) node[above]{$\g$};
			\draw[orange] (1.63,1) node {$\bullet$};
			\draw[orange] (2.25,2) node {$\bullet$};
			\draw[orange] (2.87,3) node {$\bullet$};
			\draw[orange] (3.49,4) node {$\bullet$};
			\draw[orange,line width=1pt] (3.1,3.55)--(3.26,3.45);
			\draw[orange,line width=1pt] (1.86,1.55)--(2.02,1.45);
			\draw[orange,line width=1pt] (2.48,2.55)--(2.64,2.45) node[right]{$\alpha_2\left(\g\right)$};
		\end{tikzpicture}
		\captionof{figure}{Example in dimension $d=2$: $\g$ spends the same time $\alpha_2\left(\g\right)$ between two consecutive hyperplanes of type $2$.}
	\end{center}

	We observe that
	\begin{equation}\label{temps1}
		0\leqslant \tio\leqslant\ai ; 
	\end{equation}
	and $0\leqslant T-\tik\leqslant\ai$, where $\tio$ and $\tik$ are respectively the meeting times of the first and last hyperplane of type $i$ encountered by $\g$. Moreover, $k_i\ai\leqslant T=\tio+k_i\ai+T-\tik\leqslant \left(k_i+2\right)\ai$, which implies
	\begin{equation}\label{enca}
		\frac{T}{k_i\left(1+\frac{2}{k_i}\right)}\leqslant \ai\leqslant \frac{T}{k_i} .
	\end{equation}
	Moreover, for any $j\in \llbracket 0,k_i \rrbracket$, 
	\begin{equation}\label{tij}
		\tij=\tio+j\ai .
	\end{equation}

	\begin{lemme}\label{majotmt}
		Let $\g, \g^\prime \in \G_p$. Let $i\in \llbracket 1,d \rrbracket$. For any $j\in \llbracket 0,k_i \rrbracket$, $\left| \tij - \tijp \right| \leqslant \frac{4dT}{N_p}$.
	\end{lemme}
	
	\begin{proof}
		We have $\sum\limits_{i=1}^{d}k_i=N_p$, so there exists $i_0\in\llbracket 1,d \rrbracket$ such that $k_{i_0}\geqslant \frac{N_p}{d}$. Let $i\in \llbracket 1,d \rrbracket$. Let $j\in \llbracket 0,k_i \rrbracket$, there exists $j_i\in \llbracket 0, k_{i_0} \rrbracket $ such that $\tioji\leqslant\tij\leqslant\tiojii$. Then, since $\g$ et $\g^\prime$ are in the same $\G_p$, we have $\tiojip\leqslant\tijp\leqslant\tiojiip$. By using \eqref{tij}, \eqref{temps1} and \eqref{enca}, we obtain : 
		\begin{align*}
			\left| \tioji - \tiojip \right| &= \left| \tioo + j_i \aio - \tioop - j_i \aiop \right| \leqslant \left| \tioo - \tioop \right| + j_i \left| \aio - \aiop \right| \\
			&\leqslant \frac{T}{k_{i_0}}+j_i\left(\frac{T}{k_{i_0}}-\frac{T}{k_{i_0}\left(1+\frac{2}{k_{i_0}}\right)}\right) \leqslant \frac{T}{k_{i_0}} + j_i\frac{2T}{k_{i_0}^2}\leqslant \frac{3T}{k_{i_0}}\leqslant \frac{3dT}{N_p}.
		\end{align*}
		Moreover, $\tiojii-\tioji=\aio\leqslant\frac{T}{k_{i_0}}\leqslant \frac{dT}{N_p}$.
		Hence, 
		\begin{align*}
			\tij - \tijp &\leqslant \tiojii - \tiojip \\ 
			&\leqslant \tiojii - \tioji + \tioji - \tiojip \\ 
			&\leqslant \frac{dT}{N_p}+\frac{3dT}{N_p}= \frac{4dT}{N_p}.
		\end{align*} 
		Similarly, $\tijp - \tij \leqslant \frac{4dT}{N_p}$.
	\end{proof}
	
	
	We aim to evaluate $\left|\m-\mp\right|$. To do so, we will separate, in the sum, the types of hyperplanes based on their number of occurrences along the trajectories of $\g$ and $\g'$. In the first sum, we consider those encountered less than $\sqrt{N_p}-1$ times, and in the second sum, we consider those encountered more than $\sqrt{N_p}-1$ times by the geodesics of $\G_p$. We have
	\begin{multline}\label{sumki}
		\left|\m-\mp\right| \leqslant \frac{1}{T}\sum\limits_{i=1}^{d}\left|\sum\limits_{j=0}^{k_i}\left(\tij-\tijp\right) \left(X_{ij}-Y_{ij}\right)\right|\\
		\leqslant \frac{1}{T}\sum\limits_{\substack{i\in \llbracket 1,d \rrbracket \\ k_i\leqslant \sqrt{N_p}-1}}\left|\sum\limits_{j=0}^{k_i}\left(\tij-\tijp\right) \left(X_{ij}-Y_{ij}\right)\right| + \frac{1}{T}\sum\limits_{\substack{i\in \llbracket 1,d \rrbracket \\ k_i > \sqrt{N_p}-1}}\left|\sum\limits_{j=0}^{k_i}\left(\tij-\tijp\right) \left(X_{ij}-Y_{ij}\right)\right| \\
		\leqslant \frac{1}{T}\sum\limits_{\substack{i\in \llbracket 1,d \rrbracket \\ k_i\leqslant      \sqrt{N_p}-1}}\left|\sum\limits_{j=0}^{k_i}\left(\tij-\tijp\right) \left(X_{ij}-Y_{ij}\right)\right| + \frac{1}{T}\sum\limits_{\substack{i\in \llbracket 1,d \rrbracket \\ k_i > \sqrt{N_p}-1}}\left[\left|\tio-\tiop\right|\left|\sum\limits_{j=0}^{k_i}\left(X_{ij}-Y_{ij}\right)\right| \right. \\
		\left. +\left|\ai-\aip\right|\left|\sum\limits_{j=0}^{k_i}j\left(X_{ij}-Y_{ij}\right)\right|\right] ;
	\end{multline}
	by \eqref{tij}.
	
	Let's first focus on the left-hand term of \eqref{sumki}, with the following lemma.
	
	\begin{lemme}\label{petitki}
		Let $i\in \llbracket 1,d \rrbracket$ such that $k_i\leqslant \sqrt{N_p}-1$. Then, for all $\delta>0$, if $N_p\geqslant\frac{2^4d^2}{\delta^2}$,
		\begin{equation*}
			\P\left(\frac{1}{T}\left|\sum\limits_{j=0}^{k_i}\left(\tij-\tijp\right) \left(X_{ij}-Y_{ij}\right)\right|\geqslant \delta\right)=0.
		\end{equation*}
		
	\end{lemme}
	
	\begin{proof}
		Let $\delta>0$. By the triangle inequality and Lemma \ref{majotmt}, we have:
		\begin{align*}
			\P\left(\frac{1}{T}\left|\sum\limits_{j=0}^{k_i}\left(\tij-\tijp\right) \left(X_{ij}-Y_{ij}\right)\right|\geqslant \delta\right) 
			&\leqslant \P\left(\frac{1}{T}\sum\limits_{j=0}^{k_i}\left|\tij-\tijp\right| \left|X_{ij}-Y_{ij}\right|\geqslant \delta\right)\\
			&\leqslant \P\left(\frac{4d}{N_p}\sum\limits_{j=0}^{k_i}\left|X_{ij}-Y_{ij}\right|\geqslant \delta\right).
		\end{align*}
		Now, for all $j\in \llbracket 0,k_i \rrbracket$, we have $\left|X_{ij}-Y_{ij}\right| \leqslant 1$, so
		\begin{align*}
			\P\left(\left|\sum\limits_{j=0}^{k_i}\left(\tij-\tijp\right) \left(X_{ij}-Y_{ij}\right)\right|\geqslant \delta\right) 
			&\leqslant \P\left(\frac{4d}{N_p}\left(k_i+1\right)\geqslant \delta\right)\\
			&\leqslant \P\left(\frac{4d}{N_p}\sqrt{N_p}\geqslant \delta\right) = \P\left(\frac{4d}{\sqrt{N_p}}\geqslant \delta\right)=0.\qedhere
		\end{align*}
	\end{proof}

	The following lemma allows us to address the right-hand term in the sum \eqref{sumki}.
	
	\begin{lemme}\label{grandki}
		Let $i\in \llbracket 1,d \rrbracket$. For all $\delta>0$, there exists $C_{i,\delta}>0$ such that:
		
		$\P\left(\frac{\left|\tio-\tiop\right|}{T}\left|\sum\limits_{j=0}^{k_i}\left(X_{ij}-Y_{ij}\right)\right|+\frac{\left|\ai-\aip\right|}{T}\left|\sum\limits_{j=0}^{k_i}j\left(X_{ij}-Y_{ij}\right)\right|\geqslant \delta\right) \leqslant C_{i,\delta} \exp\left(\frac{-\delta^2\left(k_i+1\right)}{2^8d^2}\right)$.
	\end{lemme}
	
	\begin{proof}
		Let $\delta>0$. We have
		\begin{multline}
			\P\left(\frac{\left|\tio-\tiop\right|}{T}\left|\sum\limits_{j=0}^{k_i}\left(X_{ij}-Y_{ij}\right)\right|+\frac{\left|\ai-\aip\right|}{T}\left|\sum\limits_{j=0}^{k_i}j\left(X_{ij}-Y_{ij}\right)\right|\geqslant \delta\right)\\
			\leqslant \P\left(\frac{\left|\tio-\tiop\right|}{T}\left|\sum\limits_{j=0}^{k_i}\left(X_{ij}-Y_{ij}\right)\right|\geqslant \frac{\delta}{2}\right)+ \P\left(\frac{\left|\ai-\aip\right|}{T}\left|\sum\limits_{j=0}^{k_i}j\left(X_{ij}-Y_{ij}\right)\right|\geqslant \frac{\delta}{2}\right). 
			\label{ineg}
		\end{multline}
		But $\left|\sum\limits_{j=0}^{k_i}\left(X_{ij}-Y_{ij}\right)\right|
		\leqslant \left|\sum\limits_{j=0}^{k_i}\left(X_{ij}-\varepsilon\right)\right| + \left|\sum\limits_{j=0}^{k_i}\left(\varepsilon - Y_{ij}\right)\right|$, and since for all $j\in \llbracket 0,k_i \rrbracket $, $X_{ij}$ et $Y_{ij}$ follow the same distribution, we obtain by Lemma \ref{majotmt} that:
		\begin{align*}
			\P\left(\frac{\left|\tio-\tiop\right|}{T}\left|\sum\limits_{j=0}^{k_i}\left(X_{ij}-Y_{ij}\right)\right|\geqslant \frac{\delta}{2}\right) 
			&\leqslant 2\P\left(\frac{\left|\tio-\tiop\right|}{T}\left|\sum\limits_{j=0}^{k_i}\left(X_{ij}-\varepsilon\right)\right|\geqslant \frac{\delta}{4}\right)\\
			&\leqslant 2\P\left(\frac{4d}{N_p}\left|\sum\limits_{j=0}^{k_i}\left(X_{ij}-\varepsilon\right)\right|\geqslant \frac{\delta}{4}\right)\\
			&\leqslant 2\P\left(\frac{4d}{k_i+1}\left|\sum\limits_{j=0}^{k_i}\left(X_{ij}-\varepsilon\right)\right|\geqslant \frac{\delta}{4}\right),
		\end{align*}
		because $N_p\geqslant k_i+1$.
		By applying the large deviation result from Proposition \ref{gd}  to the centered random variable $Y_{k_i}^\prime:=\frac{1}{k_i+1}\sum\limits_{j=0}^{k_i}\left(X_{ij}-\varepsilon\right)$, there exists $C_{i,\delta}^\prime>0$ such that $\P\left(\frac{4d}{k_i+1}\left|\sum\limits_{j=0}^{k_i}\left(X_{ij}-\varepsilon\right)\right|\geqslant \frac{\delta}{4}\right)\leqslant C_{i,\delta}^\prime \exp\left(\frac{-\delta^2(k_i+1)}{2^8d^2}\right)$. Hence,
		\begin{equation*}
			\P\left(\frac{\left|\tio-\tiop\right|}{T}\left|\sum\limits_{j=0}^{k_i}\left(X_{ij}-Y_{ij}\right)\right|\geqslant \frac{\delta}{2}\right)\leqslant 2 C_{i,\delta}^\prime \exp\left(\frac{-\delta^2(k_i+1)}{2^8d^2}\right) ;
		\end{equation*}
		which upper bounds the first term on the right-hand side of inequality \eqref{ineg}. Let's now consider the second term. 
		By similar arguments as before and using \eqref{enca} for the second inequality, we obtain:
		\begin{align*}
			\P\left(\frac{\left|\ai-\aip\right|}{T}\left|\sum\limits_{j=0}^{k_i}j\left(X_{ij}-Y_{ij}\right)\right|\geqslant \frac{\delta}{2}\right)
			&\leqslant 2\P\left(\frac{\left|\ai-\aip\right|}{T}\left|\sum\limits_{j=0}^{k_i}j\left(X_{ij}-\varepsilon\right)\right|\geqslant \frac{\delta}{4}\right)\\
			&\leqslant 2\P\left(\frac{2}{k_i\left(k_i+1\right)}\left|\sum\limits_{j=0}^{k_i}j\left(X_{ij}-\varepsilon\right)\right|\geqslant \frac{\delta}{4}\right).
		\end{align*}
		By applying the large deviation result to the centered random variable $Y_{k_i}^{\prime\prime}:=\frac{2}{k_i\left(k_i+1\right)}\sum\limits_{j=0}^{k_i}j\left(X_{ij}-\varepsilon\right)$, there exists $C_{i,\delta}^{\prime\prime}>0$ such that $\P\left(\frac{2}{k_i\left(k_i+1\right)}\left|\sum\limits_{j=0}^{k_i}j\left(X_{ij}-\varepsilon\right)\right|\geqslant \frac{\delta}{4}\right)\leqslant C_{i,\delta}^{\prime\prime} \exp\left(\frac{-\delta^2(k_i+1)}{2^4}\right)$. Hence,
		\begin{equation*}
			\P\left(\frac{\left|\ai-\aip\right|}{T}\left|\sum\limits_{j=0}^{k_i}j\left(X_{ij}-Y_{ij}\right)\right|\geqslant \frac{\delta}{2}\right)
			\leqslant 2 C_{i,\delta}^{\prime\prime} \exp\left(\frac{-\delta^2(k_i+1)}{2^4}\right).
		\end{equation*}
		Thus, \eqref{ineg} becomes:
		\begin{multline*}
			\P\left(\frac{\left|\tio-\tiop\right|}{T}\left|\sum\limits_{j=0}^{k_i}\left(X_{ij}-Y_{ij}\right)\right|+\frac{\left|\ai-\aip\right|}{T}\left|\sum\limits_{j=0}^{k_i}j\left(X_{ij}-Y_{ij}\right)\right|\geqslant \delta\right) \bigskip \\
			\leqslant 2 C_{i,\delta}^\prime \exp\left(\frac{-\delta^2(k_i+1)}{2^8d^2}\right) + 2 C_{i,\delta}^{\prime\prime} \exp\left(\frac{-\delta^2(k_i+1)}{2^4}\right)
			\leqslant 2 \left(C_{i,\delta}^\prime + C_{i,\delta}^{\prime\prime}\right) \exp\left(\frac{-\delta^2(k_i+1)}{2^8d^2}\right).\qedhere
		\end{multline*}
	\end{proof}
	

	Let $\delta>0$ and $n>\frac{2^6d^{\frac{9}{2}}}{\delta^2T}$. Then, we have $N_p\geqslant\frac{2^4d^2}{\delta^2}$. Indeed, a geodesic in $\G_p$ spends at most $\frac{\sqrt{d}}{n}$ time in the same cell. Therefore, since the elements of $\G_p$ traverse $N_p$ cells, $T$ is upper bounded by $N_p\frac{\sqrt{d}}{n}$, which implies $N_p\geqslant \frac{nT}{\sqrt{d}}$. 
	Therefore, by Lemmas \ref{petitki} and \ref{grandki}, we obtain:
	
	\begin{align*}
		& ~ \P\left(\left|\m-\mp\right|\geqslant \delta \right) \leqslant \P\left(\frac{1}{T}\sum\limits_{i=1}^{d}\left|\sum\limits_{j=0}^{k_i}\left(\tij-\tijp\right) \left(X_{ij}-Y_{ij}\right)\right|\geqslant \delta \right)\\
		\leqslant & ~ \P\left(\frac{1}{T}\sum\limits_{\substack{i\in \llbracket 1,d \rrbracket \\ k_i\leqslant \sqrt{N_p}-1}}\left|\sum\limits_{j=0}^{k_i}\left(\tij-\tijp\right) \left(X_{ij}-Y_{ij}\right)\right|\geqslant \frac{\delta}{2}\right)\\
		+ & ~ \P\left(\frac{1}{T}\sum\limits_{\substack{i\in \llbracket 1,d \rrbracket \\ k_i > \sqrt{N_p}-1}}\left[\left|\tio-\tiop\right|\left|\sum\limits_{j=0}^{k_i}\left(X_{ij}-Y_{ij}\right)\right|+\left|\ai-\aip\right|\left|\sum\limits_{j=0}^{k_i}j\left(X_{ij}-Y_{ij}\right)\right|\right]\geqslant \frac{\delta}{2}\right)\\
		\leqslant & \sum\limits_{\substack{i\in \llbracket 1,d \rrbracket \\ k_i\leqslant \sqrt{N_p}-1}}\P\left(\frac{1}{T}\left|\sum\limits_{j=0}^{k_i}\left(\tij-\tijp\right) \left(X_{ij}-Y_{ij}\right)\right|\geqslant \frac{\delta}{2d}\right)\\
		+ & \sum\limits_{\substack{i\in \llbracket 1,d \rrbracket \\ k_i > \sqrt{N_p}-1}} \P\left(\frac{1}{T}\left[\left|\tio-\tiop\right|\left|\sum\limits_{j=0}^{k_i}\left(X_{ij}-Y_{ij}\right)\right|+\left|\ai-\aip\right|\left|\sum\limits_{j=0}^{k_i}j\left(X_{ij}-Y_{ij}\right)\right|\right]\geqslant \frac{\delta}{2d}\right)\\
		\leqslant & \left(\sum\limits_{\substack{i\in \llbracket 1,d \rrbracket \\ k_i > \sqrt{N_p}-1}}C_{i, \delta} \right) \exp\left(\frac{-\delta^2\sqrt{N_p}}{2^{10}d^4}\right)
		\leqslant ~ d\left(\max\limits_{i\in \llbracket 1,d \rrbracket }C_{i, \delta} \right) \exp\left(\frac{-\delta^2\sqrt{N_p}}{2^{10}d^4}\right) ;
	\end{align*}
	by setting, for all $i\in \llbracket 1,d \rrbracket $ such that $ k_i< \sqrt{N_p}-1$, $ C_{i, \delta}=0$.
	
	Then, as $N_p\geqslant \frac{nT}{\sqrt{d}}$, we obtain, by noting $C=d\left(\max\limits_{i\in \llbracket 1,d \rrbracket }C_{i, \delta} \right)$,
	\begin{equation*}
		\P\left(\left|\m-\mp\right|\geqslant \delta \right) \leqslant C \exp\left(\frac{-\delta^2\sqrt{nT}}{2^{10}d^{\frac{17}{4}}}\right).
	\end{equation*}
	Then, since this inequality holds for all $\g, \g' \in \G_p$,  
	\begin{equation}\label{supdiff}
		\P\left(\sup\limits_{\g, \g' \in \G_p} \left|\m-\mp\right|\geqslant \delta \right) \leqslant C \exp\left(\frac{-\delta^2\sqrt{nT}}{2^{10}d^{\frac{17}{4}}}\right).
	\end{equation}

	\subsection{Conclusion}
	
	This section summarizes the previous two sections.
	
	Let $\delta>0$, we assume that $n>\frac{2^6d^{\frac{9}{2}}}{\delta^2T}$.
	We recall that for every $p\in \cP_d$, $\g_p \in \G_p$ denotes a representative of the set $\G_p$.
	
	\begin{lemme}\label{gdgp}
		For all $p\in \cP_d$, there exists $C>0$ such that $\P\left(\mpp\leqslant\varepsilon-\delta\right)\leqslant C \exp\left(\frac{-\delta^2nT^2}{d\left(T+1\right)}\right)$.
	\end{lemme}
	
	\begin{proof} Let $p\in\cP_d$.
		We have $\mpp=\sum\limits_{i=1}^{N_p}\frac{t_i\left(\g_p\right)}{T}X_i$, where for each $i\in \llbracket 1, N_p \rrbracket$, $t_i\left(\g_p\right)$ denotes the time spent by $\g_p$ in the cell $c_i$. Moreover, $\sum\limits_{i=1}^{N_p}\frac{t_i\left(\g_p\right)}{T}=1$, and for each $i\in \llbracket 1, N_p \rrbracket$, $\frac{t_i\left(\g_p\right)}{T}\leqslant \frac{\sqrt{d}\left(T+1\right)}{Tn}$. Thus, by applying Proposition \ref{gd} to the random variable $\mpp$, there exists $C>0$ such that :
		\begin{equation*}
			\P\left(\mpp\leqslant\varepsilon-\delta\right) \leqslant C \exp\left(\frac{-\delta^2N_pT}{\sqrt{d}\left(T+1\right)}\right) \leqslant C \exp\left(\frac{-\delta^2nT^2}{d\left(T+1\right)}\right).\qedhere
		\end{equation*}
	\end{proof}
	
	
	For each $p\in\cP_d$, we define $\Gamma_p = \overline{\G_p}\cap \Gamma \backslash \Gamma_H$, which is the closure of $\G_p$ in $\Gamma \backslash \Gamma_H$. Thus, we have $\G_p\subset \Gamma_p$, and $\Gamma_p$ contains the geodesics of $\G_p$, and also the geodesics that intersect at least one edge without being fully contained in a hyperplane of $\cG^n$. Therefore, $\Gamma\backslash\Gamma_H=\bigcup\limits_{p\in\cP_d}\Gamma_p$. Hence:
	
	\begin{align*}
		\P\left(\inf\limits_{\g \in \Gamma\backslash\Gamma_H} \m \leqslant\varepsilon-\delta\right) &=\P\left(\inf\limits_{p \in \cP_d}\left( \inf\limits_{\g \in \Gamma_p} \m \right) \leqslant\varepsilon-\delta\right)\\
		&\leqslant \sum\limits_{p\in\cP_d} \P\left(\inf\limits_{\g \in \Gamma_p} \m \leqslant\varepsilon-\delta\right).
	\end{align*}
	
	Let $p\in\cP_d$, let's evaluate $\P\left(\inf\limits_{\g \in \Gamma_p} \m \leqslant\varepsilon-\delta\right)$.
	By continuity of $\g \mapsto \m$ on $\Gamma \backslash \Gamma_H$, we have 
	\begin{equation*}
		\P\left(\inf\limits_{\g \in \Gamma_p} \m \leqslant\varepsilon-\delta\right) = \P\left(\inf\limits_{\g \in \G_p} \m \leqslant\varepsilon-\delta\right).
	\end{equation*}
	Then, 
	
	\begin{align*}
		&\P\left(\inf\limits_{\g \in \G_p} \m \leqslant\varepsilon-\delta\right) = \P\left(\inf\limits_{\g \in \G_p} \m - \mpp + \mpp \leqslant\varepsilon-\delta\right)\\
		&= \P\left(\inf\limits_{\g \in \G_p} \left( \m - \mpp \right) + \mpp \leqslant\varepsilon-\delta \; \middle| \; \mpp \leqslant \varepsilon - \frac{\delta}{2} \right) \P\left(\mpp \leqslant \varepsilon - \frac{\delta}{2} \right)\\
		&+ \, \P\left(\inf\limits_{\g \in \G_p} \left( \m - \mpp \right) + \mpp \leqslant\varepsilon-\delta \; \middle| \; \mpp > \varepsilon - \frac{\delta}{2} \right) \P\left(\mpp > \varepsilon - \frac{\delta}{2} \right) \\
		&\leqslant \P\left(\mpp \leqslant \varepsilon - \frac{\delta}{2} \right) + \, \P\left(\inf\limits_{\g \in \G_p} \left( \m - \mpp \right) \leqslant -\frac{\delta}{2} \right) \\
		&\leqslant \P\left(\mpp \leqslant \varepsilon - \frac{\delta}{2} \right) + \, \P\left(\left|\inf\limits_{\g \in \G_p} \left( \m - \mpp \right)\right| \geqslant \frac{\delta}{2} \right) \\
		&\leqslant \P\left(\mpp \leqslant \varepsilon - \frac{\delta}{2} \right) + \, \P\left(\sup\limits_{\g \in \G_p} \left| \m - \mpp \right| \geqslant \frac{\delta}{2} \right).
	\end{align*}

	We then upper bound each term using the results previously obtained: the left-hand side term using Lemma \ref{gdgp}, and the right-hand side term using \eqref{supdiff}. Thus, there exist $C_1>0$ and $C_2>0$ such that:
	\begin{equation*}
		\P\left(\inf\limits_{\g \in \G_p} \m \leqslant\varepsilon-\delta\right) \leqslant C_1\exp\left(\frac{-\delta^2nT^2}{4d\left(T+1\right)}\right)+ C_2 \exp\left(\frac{-\delta^2\sqrt{nT}}{2^{12}d^{\frac{17}{4}}}\right).
	\end{equation*}
	Thus, 
	\begin{align*}
		\P\left(\inf\limits_{\g \in \Gamma\backslash\Gamma_H} \m \leqslant\varepsilon-\delta\right) &\leqslant \sum\limits_{p\in\cP_d} \P\left(\inf\limits_{\g \in \G_p} \m \leqslant\varepsilon-\delta\right)\\
		&\leqslant \sum\limits_{p\in\cP_d} \left[C_1\exp\left(\frac{-\delta^2nT^2}{4d\left(T+1\right)}\right)+ C_2 \exp\left(\frac{-\delta^2\sqrt{nT}}{2^{12}d^{\frac{17}{4}}}\right)\right]\\
		&\leqslant \#\cP_d \left[C_1\exp\left(\frac{-\delta^2nT^2}{4d\left(T+1\right)}\right)+ C_2 \exp\left(\frac{-\delta^2\sqrt{nT}}{2^{12}d^{\frac{17}{4}}}\right)\right].
	\end{align*}
	Finally, by Lemma \ref{ppoly}, since $\#\cP_d$ is polynomial in $n$, 
	
	\begin{equation*}
		\underset{n \to +\infty}{\lim}\P\left(\inf\limits_{\g \in \Gamma\backslash\Gamma_H} \m \leqslant\varepsilon-\delta\right) = 0.
	\end{equation*}

	\appendix
	
	\section{An arithmetic problem}
	
	Let $a_1,\dots,a_n \in \R^2$ such that $a_i\neq a_j$ if $i\neq j$. Let $\cA=\left\{a_1, a_2, \dots, a_{n}\right\}$ the set formed by these points.
	We define 
	\begin{equation*}
		\Lambda\left(\cA\right)=\left\{\left\{S,\cA \backslash S \right\} \middle| S\in\cP(\cA) \text{ such that there exists a line } \g \text{ such that } \g\cap \cA=\emptyset \text{ and } \g \text{ splits } S \text{ and } \cA \backslash S \right\}.
	\end{equation*}
	
	\begin{lemme}\label{sepnpt}
		There exists $P\in \R[X]$ such that for any set of $n$ points in the plane $\cA=\left\{a_1,\dots,a_n\right\}$, $\#\Lambda\left(\cA\right)\leqslant P(n)$, i.e. the number of ways to strictly separate $n$ points with a line in the plane is polynomial in $n$.
	\end{lemme}
	
	\begin{proof}
		Consider a set $\cA=\left\{a_1, a_2, \dots, a_{n}\right\}$ of $n$ points in the plane.
		The directions are parameterized by the real projective line $\P^1(\R)$, which is identified with $[0,\pi[$ through the angle formed by a direction and the x-axis.
		For any $i,j\in\llbracket 1, n\rrbracket$, let $\theta(i,j)\in\P^1(\R)$ be the direction of the line passing through $a_i$ and $a_j$. Let $\cD(\cA)=\left\{\theta(i,j)\in \P^1(\R) \, \middle| \, i,j\in\llbracket 1, n\rrbracket\right\}$, be the set of directions formed by two points in $\cA$. Then $N:=\#\cD(\cA)\leqslant \binom{n}{2}=\frac{n(n-1)}{2}$.
		Let $0\leqslant\theta_1<\theta_2<\dots<\theta_N<\pi$ be the elements of $\cD(\cA)$.
		Let $\gamma$ be a line in the plane such that $\g\cap\cA=\emptyset$. Then, $\gamma$ splits $\mathcal{A}$ into two sets, $S_1(\gamma)$ and $\mathcal{A}\backslash S_1(\gamma)$ where $S_1(\gamma)$ is the one that contains $a_1$. We also identify the partition $\{S_1(\gamma),\mathcal{A}\backslash S_1(\gamma)\}$ with the set $S_1(\gamma)$ that contains $a_1$.
		For any $\theta \in \P^1(\R)$, let $P(\theta)=\left\{S_1(\g) \, \middle| \, \g \text{ line with slope } \theta \right\}$ be the set of partitions of points in $\cA$ according to the direction $\theta$.
		Let $i\in\llbracket 1, N-1\rrbracket$.
		Let $\theta, \theta' \in ]\theta_i,\theta_{i+1}[$, let us show that $P(\theta)=P(\theta')$.
		Let $A\in P(\theta)$, then there exists $\g$ a line with slope $\theta$ such that $A=S_1(\g)$. Let us show that $A\in P(\theta')$, i.e., there exists $\g'$ a line with slope $\theta'$ such that $A=S_1(\g')$. 
		We translate $\gamma$ orthogonally until it intersects a point of $\mathcal{A}$. This transformation does not alter the slope of the resulting line. Let $\tilde{\gamma}$ be the resulting line. Then, there exists a single point $a\in\mathcal{A}$ such that $a\in\tilde{\gamma}$ ($a$ is unique since $\theta \notin \cD(\cA)$). We then rotate $\tilde{\gamma}$ with center $a$ until it has slope $\theta'$. 
		Note that during this rotation, the line with slope $\tilde{\theta}\in[\theta, \theta']$ does not intersect any other points of set $\cA$, since $\tilde{\theta}\notin \cD(\cA)$. 
		Let $\tilde{\tilde{\gamma}}$ be the resulting line. 
		The line $\tilde{\tilde{\gamma}}$ contains $a$ and splits the other points into:
		\begin{itemize}
			\item $(S_1(\g)\backslash \{a\},\mathcal{A}\backslash S_1(\g))$ if $a\in S_1(\g)$;
			\item $(S_1(\g),\mathcal{A}\backslash (S_1(\g)\cup \{a\}))$ if $a\notin S_1(\g)$.
		\end{itemize}
		By slightly othogonally translating $\tilde{\tilde{\gamma}}$ in the good direction, we obtain a line $\gamma'$ with slope $\theta'$ such that $\g'\cap\cA=\emptyset$ and $S_1(\gamma')=A$.
		
		Furthermore, for any $i\in\llbracket 1, N-1\rrbracket$ and $\theta \in ]\theta_i,\theta_{i+1}[$, $\# P(\theta)\leqslant n+1$ (by translating a line with slope $\theta$ in the plane, we sweep through the set of $n$ points of $\cA$).
		
		Thus, $\#\Lambda\left(\cA\right)\leqslant \left(N+1\right)\left(n+1\right) \leqslant \left(\frac{n\left(n-1\right)}{2} +1\right)\left(n+1\right)$.
		\begin{center}
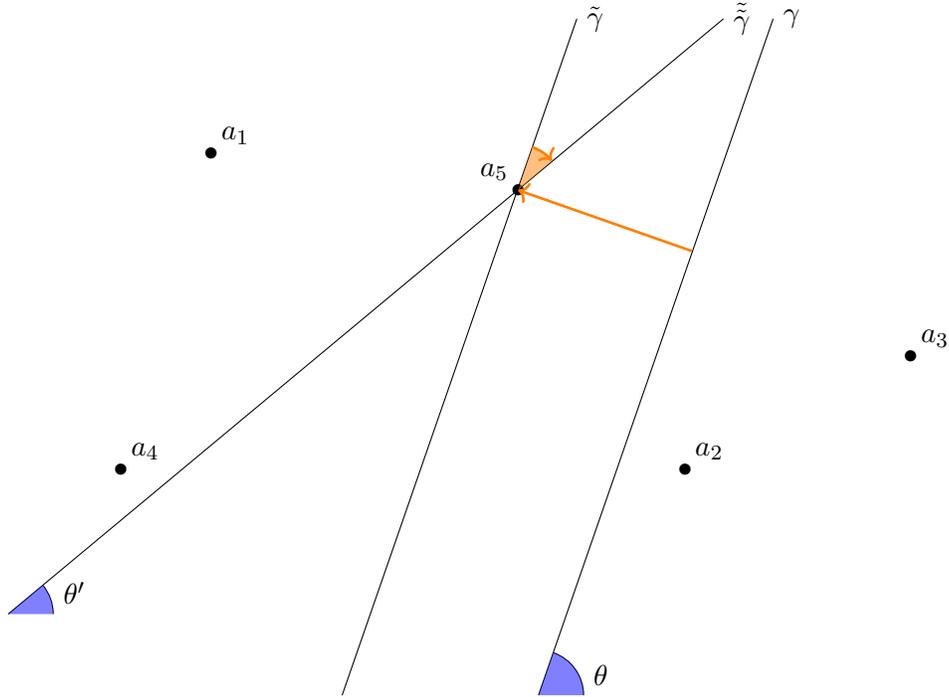

			\begin{tikzpicture}[scale=3]
				\draw (0.9,3.4) node{$\bullet$} node[above right]{$a_1$};
				\draw (3,2) node{$\bullet$} node[above right]{$a_2$};
				\draw (4,2.5) node{$\bullet$} node[above right]{$a_3$};
				\draw (0.5,2) node{$\bullet$} node[above right]{$a_4$};
				\draw (2.26,3.24) node{$\bullet$} node[above left]{$a_5$};
				\draw (3.39,4)node[right]{$\g$}--(2.35,1);
				\draw (2.52,4)node[right]{$\tilde{\g}$}--(1.48,1);
				\draw (3.17,4)node[right]{$\tilde{\tilde{\g}}$}--(0,1.36);
				\draw[->,orange,line width=1pt] (3.03,2.97)--(2.26,3.24);
				\draw (2.55,1) node[above right]{$\theta$} arc (0:70.86:0.2);
				\draw (0.2,1.36) node[above right]{$\theta'$} arc (0:39.81:0.2);
				\fill[blue,opacity=0.5] (2.35,1)--(2.55,1) arc (0:70.86:0.2)--(2.35,1);
				\fill[blue,opacity=0.5] (0,1.36)--(0.2,1.36) arc (0:39.81:0.2)--(0,1.36);
				\draw[<-,orange,line width=1pt] (2.41,3.37) arc (39.81:70.86:0.2);
				\fill[orange,opacity=0.5] (2.26,3.24)--(2.41,3.37) arc (39.81:70.86:0.2)--(2.26,3.24);
			\end{tikzpicture}
			\captionof{figure}{Example of a set $\cA=\{a_1,\dots,a_5\}$ and representation of $\tilde{\g}$ and $\tilde{\tilde{\g}}$ for $\g$, $\theta$ and $\theta'$ fixed}\label{espproj}
		\end{center}
		
	\end{proof}

\end{document}